\documentclass{amsart}
\usepackage{amsfonts, amssymb, wasysym}
\usepackage{multicol}
\usepackage{amsmath}
\usepackage{latexsym}
\usepackage{mathrsfs}
\usepackage{epsfig}
\usepackage{graphicx}
\usepackage{tikz}
\usetikzlibrary{positioning}
\usepackage{hyperref}
\hypersetup{colorlinks,citecolor=blue,linkcolor=blue}
\usepackage{slashed}
\usepackage[textsize=scriptsize]{todonotes}
\usepackage{hyperref}
\newtheorem{theorem}{Theorem}[section]

\newtheorem{lemma}[theorem]{Lemma}

\def\diag{\mathrm{diag}}

\usepackage{graphicx}
\makeatletter
\DeclareRobustCommand*\uell{\mathpalette\@uell\relax}
\newcommand*\@uell[2]{
  \setbox0=\hbox{$#1\ell$}
  \setbox1=\hbox{\rotatebox{10}{$#1\ell$}}
  \dimen0=\wd0 \advance\dimen0 by -\wd1 \divide\dimen0 by 2
  \mathord{\lower 0.1ex \hbox{\kern\dimen0\unhbox1\kern\dimen0}}
}

\begin{document}
\title{The $\hat G$-Index of a Spin, Closed, Hyperbolic Manifold of Dimension 2 or 4}

\author{John G. Ratcliffe and Steven T. Tschantz}

\address{Department of Mathematics, Vanderbilt University\\ 
Nashville, TN 37240, USA\\ }
\email{j.g.ratcliffe@vanderbilt.edu}

\begin{abstract}
In this paper, we develop general techniques for computing the Atiyah-Singer $\hat G$-index of a spin, closed, hyperbolic 2- or 4-manifold, 
and apply these techniques to compute the $\hat G$-index of the fully symmetric spin structure of the Davis hyperbolic 4-manifold. 
\end{abstract}

\subjclass{57M50, 58J20, 53C27}

\keywords{hyperbolic $4$-manifold,  harmonic spinor, Davis manifold}

\maketitle

\section{Introduction}\label{S:intro} 

Let $M = \Gamma\backslash H^n$ be a spin, closed, hyperbolic $n$-manifold. 
This means that the discrete subgroup $\Gamma$ of ${\rm SO}^+(n,1)$ lifts to a subgroup $\hat\Gamma$ of ${\rm Spin}^+(n,1)$ such that the 
double-covering epimorphism $\eta: {\rm Spin}^+(n,1) \to {\rm SO}^+(n,1)$ maps $\hat\Gamma$ isometrically onto $\Gamma$, 
thus giving a spin structure  $\hat\Gamma\backslash {\rm Spin}^+(n,1)$ on $M$ by Theorem 2.1 of \cite{RRT}.

Let $N(\Gamma)$ be the normalizer of $\Gamma$ in ${\rm SO}^+(n,1)$. 
The group of orientation-preserving isometries of $M$ is represented by the quotient group $\Gamma\backslash N(\Gamma)$, that is, 
if $f$ is in $N(\Gamma)$, then $f_\star: M \to M$, defined by $f_\star(\Gamma x) = \Gamma fx$, is an orientation-preserving isometry of $M$,  
and every orientation-preserving isometry of $M$ is of this form. 

Likewise, the group of symmetries of the spin structure $\hat\Gamma\backslash {\rm Spin}^+(n,1)$ is represented by the quotient group 
$\hat\Gamma\backslash N(\hat\Gamma)$ where $N(\hat\Gamma)$ is the normalizer of $\hat\Gamma$ in $ {\rm Spin}^+(n,1)$. 
If $\hat f$ in $N(\hat\Gamma)$, then $\hat f$ induces a symmetry $\hat f_\star$ of the spin structure 
$\hat\Gamma\backslash {\rm Spin}^+(n,1)$ defined by $\hat f_\star(\hat\Gamma \hat g) = \hat\Gamma(\hat f\hat g)$, 
and every symmetry of $\hat\Gamma\backslash {\rm Spin}^+(n,1)$ is of this form by Theorem 2.2 of \cite{RRT}.

Let $\hat G$ be the group of symmetries of the spin structure $\hat\Gamma\backslash {\rm Spin}^+(n,1)$ of $M$. 
Then the double-covering epimorphism  $\eta: {\rm Spin}^+(n,1) \to {\rm SO}^+(n,1)$ induces a double-covering epimorphism 
from $\hat G$ to the group $G$ of all the orientation-preserving isometries of $M$ that lift to a symmetry of the spin structure $\hat\Gamma\backslash {\rm Spin}^+(n,1)$. 
We say that the spin structure $\hat\Gamma\backslash {\rm Spin}^+(n,1)$ is {\it fully symmetric} if $G$ is the full group of orientation-preserving 
isometries of $M$. 

Assume that $n$ is even, then $\hat G$ acts on the finite-dimensional, complex, vector spaces $\mathcal H^+$ and $\mathcal H^-$ of positive and negative harmonic spinors on $M$, cf. \cite{hitchin:spinors, L-M}. This gives two representations $\rho^+$ and $\rho^-$ of the finite group $\hat G$ whose difference $\rho^+ - \rho^-$ in the representation ring $R(\hat G)$ is the $\hat G$-index, ${\rm Spin}(\hat G, M)$, of the action of $\hat G$ on 
the space $\mathcal H = \mathcal H^+\oplus\mathcal H^-$ of harmonic spinors of the spin structure of $M$.

The Atiyah-Singer $G$-Spin theorem \cite{A-H,atiyah-singer:III, L-M,S} gives a formula for ${\rm Spin}(\hat G, M)$ 
which unfortunately involves indeterminate sign terms. 
In our previous paper \cite{RRT}, we developed techniques to determine some of these sign terms but not all of them. 
In this paper, we develop general techniques for computing ${\rm Spin}(\hat G, M)$ for $n = 2, 4$, 
and we apply our new techniques to compute the $\hat G$-index of the fully symmetric spin structure 
of the Davis hyperbolic 4-manifold $M$. 

Our paper is organized as follows: 
In Section 2, we describe the representation of ${\rm Spin}^+(4,1)$ by the matrix group ${\rm SU}(1,1;\mathbb H)$. 
In Section 3, we show that the Davis hyperbolic 4-manifold has a unique fully symmetric spin structure 
whose group of symmetries $\hat G$ has order $28,800$. 
In Section 4, we determine the structure of the group $\hat G$ in terms of the binary icosahedral group $2I$. 
In Section 5, we determine all the irreducible representations of $\hat G$ in terms of the irreducible representations 
of $2I$. 
In Section 6, we prove some results concerning totally geodesic surfaces in an  orientable, closed, hyperbolic 4-manifold. 
In Section 7, we develop our general techniques for computing the $\hat G$-index of a spin, closed, hyperbolic 4-manifold. 
In Section 8, we compute the $\hat G$-index of the fully symmetric spin structure of the Davis hyperbolic 4-manifold $M$. 
In Section 9, we show that ${\rm Spin}(\hat G,M)$ is the difference of two 12-dimensional irreducible representations of $\hat G$  in $R(\hat G)$.
In Section 10, we develop our techniques for computing the $\hat G$-index of a spin, closed, hyperbolic 2-manifold.

\section{The Representation of ${\rm Spin}^+(4,1)$ by ${\rm SU}(1,1;\mathbb{H})$}\label{S:2}  

Let $\mathbb{H}$ be the ring of quaternions, and let $\mathbb H(2)$ be the algebra of $2\times 2$ matrices over $\mathbb H$. 
If $A$ is in $\mathbb H(2)$, let $A^*$ be the conjugate transpose of $A$.  Let $J = {\rm diag}(1,-1)$, and let 
$${\rm SU}(1,1;\mathbb H) = \{A \in \mathbb H(2): A^*JA = J\}.$$
The group ${\rm SU}(1,1;\mathbb H)$ acts on the conformal ball model of hyperbolic 4-space
$$B^4 =\{q \in \mathbb H: |q| < 1\}$$
by linear fractional transformations so that 
$$\left(\begin{array}{cc} a & b \\ c & d\end{array}\right)\cdot q = (aq+b)(cq +d)^{-1}.$$

In \S6 of our paper \cite{RRT}, we defined a double-covering epimorphism 
$$\eta: {\rm SU}(1,1;\mathbb H) \to {\rm SO}^+(4,1).$$
We now make a small change in the definition of $\eta$ for an aesthetic reason. 
We change the definition of $E_5$ from $J$ to $-J$, and we replace $\rho^+$ by $\rho^-$ throughout \S6 of \cite{RRT}. 
The new $\eta$ is conjugation by $J$ followed by the old $\eta$.
All the main results of \S6 and \S10 of \cite{RRT} are unaltered by this change. 
The only consequence of this change to \cite{RRT} is that the off-diagonal entries of a lift $\hat M$ in ${\rm SU}(1,1;\mathbb H)$ 
of a matrix $M$ in ${\rm SO}^+(4,1)$, with respect to $\eta$, change sign. 

With this change, $\eta$ is now compatible with stereographic projection $\zeta: B^4 \to H^4$ (cf. Formula 4.5.2 \cite{R}), that is, for all $A \in {\rm SU}(1,1;\mathbb H)$ and all $q \in B^4$, we have 
$$\zeta(A\cdot q) = \eta(A)\zeta(q).$$
It is now obvious that ${\rm SU}(1,1;\mathbb H)$ acts on $B^4$ via orientation-preserving isometries, since $\zeta: B^4 \to H^4$ is an isometry. 

If $q \in \mathbb{H}$, write $q = q_0 +q_1{\bf i} + q_2{\bf j} + q_3{\bf k}$ with $q_i$ real for $i=0,1,2,3$. 
If $A$ in ${\rm SU}(1,1;\mathbb{H})$ has row vectors $(a,b)$ and $(c,d)$,
then the new definition of $\eta$ is given by $\eta(A) = (m_{ij})$ where $m_{ij}$ is listed in Table \ref{ta:1}. 
The last row and last column of $\eta(A)$ in Table \ref{ta:1} appear to be unbalanced with respect to the entries of $A$; however, 
this is not the case by Lemma \ref{L:2.1}.

\begin{table} 
\begin{eqnarray*}
m_{11} &= & b_0c_0+b_1c_1+b_2c_2+b_3c_3+a_0d_0+a_1d_1+a_2d_2+a_3d_3 \\
m_{12} & = & b_1c_0-b_0c_1-b_3c_2+b_2c_3-a_1d_0+a_0d_1+a_3d_2-a_2d_3 \\
m_{13} & = & b_2c_0+b_3c_1-b_0c_2-b_1c_3-a_2d_0-a_3d_1+a_0d_2+a_1d_3 \\
m_{14} & = & b_3c_0-b_2c_1+b_1c_2-b_0c_3-a_3d_0+a_2d_1-a_1d_2+a_0d_3 \\
m_{15} & = & 2b_0d_0+2b_1d_1+2b_2d_2+2b_3d_3 \\
m_{21} & = & b_1c_0-b_0c_1+b_3c_2-b_2c_3+a_1d_0-a_0d_1+a_3d_2-a_2d_3 \\
m_{22} & = & -b_0c_0-b_1c_1+b_2c_2+b_3c_3+a_0d_0+a_1d_1-a_2d_2-a_3d_3 \\
m_{23} & = & b_3c_0-b_2c_1-b_1c_2+b_0c_3-a_3d_0+a_2d_1+a_1d_2-a_0d_3 \\
m_{24} & = & -b_2c_0-b_3c_1-b_0c_2-b_1c_3+a_2d_0+a_3d_1+a_0d_2+a_1d_3 \\
m_{25} & = & 2b_1d_0-2b_0d_1+2b_3d_2-2b_2d_3 \\
m_{31} & = & b_2c_0-b_3c_1-b_0c_2+b_1c_3+a_2d_0-a_3d_1-a_0d_2+a_1d_3 \\
m_{32} & = & -b_3c_0-b_2c_1-b_1c_2-b_0c_3+a_3d_0+a_2d_1+a_1d_2+a_0d_3 \\
m_{33} & = & -b_0c_0+b_1c_1-b_2c_2+b_3c_3+a_0d_0-a_1d_1+a_2d_2-a_3d_3 \\
m_{34} & = & b_1c_0+b_0c_1-b_3c_2-b_2c_3-a_1d_0-a_0d_1+a_3d_2+a_2d_3 \\
m_{35} & = & 2b_2d_0-2b_3d_1-2b_0d_2+2b_1d_3 \\
m_{41} & = & b_3c_0+b_2c_1-b_1c_2-b_0c_3+a_3d_0+a_2d_1-a_1d_2-a_0d_3 \\
m_{42} & = & b_2c_0-b_3c_1+b_0c_2-b_1c_3-a_2d_0+a_3d_1-a_0d_2+a_1d_3 \\
m_{43} & = & -b_1c_0-b_0c_1-b_3c_2-b_2c_3+a_1d_0+a_0d_1+a_3d_2+a_2d_3 \\
m_{44} & = & -b_0c_0+b_1c_1+b_2c_2-b_3c_3+a_0d_0-a_1d_1-a_2d_2+a_3d_3 \\
m_{45} & = & 2b_3d_0+2b_2d_1-2b_1d_2-2b_0d_3 \\
m_{51} & = & 2a_0b_0+2a_1b_1+2a_2b_2+2a_3b_3 \\
m_{52} & = & -2a_1b_0+2a_0b_1+2a_3b_2-2a_2b_3 \\
m_{53} & = & -2a_2b_0-2a_3b_1+2a_0b_2+2a_1b_3 \\
m_{54} & = & -2a_3b_0+2a_2b_1-2a_1b_2+2a_0b_3 \\
m_{55} & = & a_0^2+a_1^2+a_2^2+a_3^2+b_0^2+b_1^2+b_2^2+b_3^2.
\end{eqnarray*}

\vspace{.15in}
\caption{The double-covering epimorphism $\eta:{\rm SU(1,1;\mathbb{H}) \to  {\rm SO}^+(4,1)}$}\label{ta:1}
\end{table}

\begin{lemma}\label{L:2.1} 
Let $A$ be a matrix in ${\rm SU}(1,1;\mathbb H)$ with row vectors $(a,b)$ and $(c,d)$, then $|a| = |d|$, $|b| = |c|$, $\overline b a = \overline d c$, 
and $c \overline a = d\overline b$. 
\end{lemma}
\begin{proof}
By Lemma 6.3 of \cite{RRT}, we have $|a|^2-|c|^2 = 1 = |d|^2 - |b|^2$, and $\overline b a = \overline d c$.
The inverse of $A$ has row vectors $(\overline a, -\overline c)$ and $(-\overline b,\overline d)$, and so 
we have $|a|^2 -|b|^2 = 1$ and $c \overline a = d\overline b$.  Hence $|b|^2 = |c|^2$ and $|a|^2 = |d|^2$. 
\end{proof}

To prove that $\zeta (A \cdot q) = \eta(A)\zeta(q)$, we verified that $A\cdot q = \zeta^{-1}(\eta(A)\zeta(q))$. 
For example, for $q = 0$, we have
$A \cdot 0  = bd^{-1} = b\overline d/|d|^2.$
By Lemma 6.3 of \cite{RRT} and Lemma \ref{L:2.1}, we have $1+ |a|^2+|b|^2 = 2|d|^2.$
By Formula 4.5.3 of \cite{R}, we have
$$\zeta^{-1}(\eta(A)\zeta(0)) =  \zeta^{-1}(\eta(A)e_5) = \zeta^{-1}(m_{15},\ldots,m_{55}) =  b\overline d/|d|^2.$$
We leave the proof for $q \neq 0$ as an exercise for the reader. 

\section{The Fully Symmetric Spin Structure on the Davis 4-Manifold}\label{S:3}  

The Davis hyperbolic 4-manifold $M$, first defined in \cite{davis:hyperbolic}, 
is obtained by gluing the opposite sides of a regular hyperbolic 120-cell $\mathcal C$, with dihedral angle $2\pi/5$, 
by translations with axes through the center of $\mathcal C$. The 120-cell $\mathcal C$ is subdivided by barycentric subdivision into 14,400 copies of the $(5,3,3,5)$ hyperbolic Coxeter 4-simplex $\Delta$, and so $\Delta$ is the basic building block of $M$. 
The Coxeter diagram of $\Delta$ reads the same forwards as backwards, and so $\Delta$ has an isometric involution $\sigma$, which is a rotation of $180$ degrees that fixes a 2-dimensional plane of $H^4$.  

In our paper \cite{ratcliffe-tschantz:davis}, 
we proved that $\sigma$ induces an orientation-preserving isometric involution $\sigma_\star$ of $M = \Gamma\backslash H^4$ by showing that $\sigma$ normalizes $\Gamma$. 
The isometry $\sigma_\star$ is then well defined by the formula 
$\sigma_\star(\Gamma x) = \Gamma\sigma x$ for each $x$ in $H^4$. 
We call $\sigma_\star$ the ``inside-out" isometry of $M$ because $\sigma_\star$ interchanges the two points $C$ and $A$ of $M$ represented by the center of  the 120-cell fundamental domain $\mathcal C$ of $\Gamma$ and the cycle of the vertices of $P$. 
This pair of points of $M$ is canonical, since every isometry of $M$ either fixes the points $C$ and $A$ or interchanges them. 

If $\phi$ is a symmetry of $\mathcal C$, then $\phi$ induces an isometry $\phi_\star$ of $M$ defined by $\phi_\star(\Gamma x) =\Gamma \phi x$ for each $x$ in $H^4$. 
The next theorem follows directly from Theorem 3 of \cite{ratcliffe-tschantz:davis}. 

\begin{theorem}\label{T:3.1} 
Let $G$ be the group of orientation-preserving isometries of the Davis $4$-manifold $M$, and let
${\rm Sym}(\mathcal C)_0$ be the group of orientation-preserving symmetries of the regular $120$-cell\, $\mathcal{C}$, of order $7200$.  
Then we have a split short exact sequence
$$1 \to {\rm Sym}_0(\mathcal C) \to G \to \langle \sigma_\star \rangle \to 1$$
where the injection is defined by $\phi \mapsto \phi_\star$, and $\sigma_\star$ maps to itself by the projection. 
Therefore, the order of $G$ is $14,400$.
\end{theorem}

\begin{theorem}\label{T:3.2} 
The Davis hyperbolic 4-manifold $M = \Gamma\backslash H^4$ has a unique fully symmetric spin structure. 
\end{theorem}
\begin{proof}
In our paper \cite{RRT}, we described a unique lift of $\Gamma$ to a subgroup $\hat\Gamma$ of ${\rm Spin}^+(4,1)$ 
such that $\widehat{\rm Sym}(\mathcal{C})_0 = \eta^{-1}({\rm Sym}(\mathcal{C})_0)$ normalizes $\hat\Gamma$.  
We show that the spin structure $\hat\Gamma\backslash {\rm Spin}^+(4,1)$ of $M$ is fully symmetric 
by showing that $\sigma$ lifts to an element $\hat\sigma$ of ${\rm Spin}^+(4,1)$ that also normalizes $\hat \Gamma$.

Let $\tau = (1+\sqrt{5})/2$ be the golden ratio, and let $\kappa = \sqrt{1+3\tau}$. 
The isometric involution $\sigma$ of $H^4$ is represented by the following matrix in ${\rm SO}^+(4,1)$
$$\left(\begin{array}{ccccc}
-4-7\tau & -1-3\tau & 0 & -1-\tau & (2+3\tau)\kappa \\
-1-3\tau & -1-\tau  & 0 & 0 & (1+\tau)\kappa \\
0 & 0 & 1 & 0 & 0 \\
-1-\tau & 0 & 0 & 0 & \kappa \\
-(2+3\tau)\kappa & -(1+\tau)\kappa & 0 & -\kappa & 5 + 8\tau 
\end{array}\right).$$
The isometric involution $\sigma$ lifts, with respect to $\eta: {\rm SU}(1,1;\mathbb{H}) \to {\rm SO}^+(4,1)$, to 
the following matrix of order 4, 

$$\hat\sigma = \frac{\sqrt{\tau -1}}{2}\left(\begin{array}{cc} 
(1+\tau)\kappa {\bf i} +\kappa{\bf j} -\tau\kappa{\bf k} & \tau -(1+3\tau){\bf i} +(1+2\tau){\bf k} \\
\tau +(1+3\tau){\bf i} -(1+2\tau){\bf k} & -(1+\tau)\kappa{\bf i} +\kappa{\bf j} + \tau\kappa{\bf k} \end{array}\right).$$

In our paper \cite{ratcliffe-tschantz:davis}, we give 120 generators $g_1,\ldots, g_{120}$ 
of the Davis manifold group $\Gamma$ that represent the side-pairing maps of the 120-cell fundamental domain $\mathcal C$ of $\Gamma$.  In \cite{ratcliffe-tschantz:davis}, we give an indirect proof that $\sigma$ normalizes $\Gamma$. 
To give a direct proof, requires writing $\sigma g_i\sigma^{-1}$ as a word $w_i$ in the generators $g_1,\ldots, g_{120}$ for each $i$. This is not hard to do by a computer calculation. For example,
$$\sigma g_1\sigma^{-1} = g_5g_{84}g_{48}g_{69}g_{39}g_{114}g_6g_{83}g_{51}g_{74}g_{36}g_{117}.$$
In \S 10.3 of \cite{RRT}, we lift the generators $g_1, \ldots, g_{120}$ to generators 
$\hat g_1, \ldots, \hat g_{120}$ of $\hat\Gamma$ so that $\widehat{\rm Sym}(\mathcal C)_0$  
acts transitively on $\hat g_1, \ldots, \hat g_{120}$ by conjugation. 
It turns out that $\hat\sigma \hat g_i \hat\sigma^{-1}$ is equal to the same word $w_i$ in the generators $\hat g_1,\ldots, \hat g_{120}$ of $\hat\Gamma$ for each $i$, and so $\hat \sigma$ also normalizes $\hat\Gamma$. 
Thus, the spin structure $\hat\Gamma\backslash{\rm Spin}^+(4,1)$ is fully symmetric by Theorem \ref{T:3.1}.
\end{proof}

\section{The Structure of the Group $\hat G$}\label{S:4}  

In this section, we describe the structure of the group $\hat G$ of symmetries of the fully symmetric spin structure $\hat\Gamma\backslash {\rm Spin}^+(4,1)$ of the Davis manifold $M$ that we considered in \S \ref{S:3}. 
An element $\hat\Gamma \hat f$ of $\hat\Gamma\backslash N(\hat\Gamma)$ represents the symmetry $\hat f_\star$ 
of $\hat\Gamma\backslash{\rm Spin}^+(4,1)$ defined by $\hat f_\star(\hat\Gamma \hat g) = \hat\Gamma (\hat f\hat g)$. 
The projection of $N(\hat\Gamma)$ onto $N(\Gamma)$ via $\eta$ induces an epimorphism from $\hat G$ to $G$, 
defined by $\hat f_\star \mapsto f_\star$, and 
a nonsplit short exact sequence
$$1\to \{\pm 1\} \to \hat G \to G \to 1.$$
The inclusion of $\widehat{\rm Sym}(\mathcal C)_0$ into $N(\hat\Gamma)$ induces a short exact sequence
$$1 \to \widehat{\rm Sym}(\mathcal C)_0 \to \hat G \to \langle \sigma_\star \rangle \to 1.$$
We have a nonsplit short exact sequence
$$1\to \{\pm 1\} \to \widehat{\rm Sym}(\mathcal C)_0 \to {\rm Sym}(\mathcal C)_0 \to 1.$$
The group ${\rm Sym}(\mathcal C)_0$ is the group of orientation-preserving symmetries 
of the regular 120-cell $\mathcal C$. The group ${\rm Sym}(\mathcal C)_0$ has two isomorphic subgroups $A$ and $B$ 
such that ${\rm Sym}(\mathcal C)_0 = A\times _{\{\pm I\}} B$, 
that is, ${\rm Sym}(\mathcal C)_0 = AB$, every element of $A$ commutes with every element of $B$, and  $A\cap B = \{\pm I\}$. 

We now define $A$ and $B$ geometrically. Let $S$ be side of $\mathcal C$. 
The side $S$ is a regular dodecahedron.  Let $F$ be a pentagonal face of $S$, 
and let $E$ be an edge of $F$. Let $F'$ be the opposite face of $S$, 
and let $S'$ be the adjacent side of $\mathcal C$ such that $S'\cap S = F'$. 
Looking across $F$ through $S$, one sees that $F'$ does not line up with $F$, cf. Figure 10.1.1 of \cite{R}. 
The face $F'$ is a rotated image of $F$ by an angle of $\pi/5$.  Let $E'$ be the edge of $F'$ 
that is the rotated image of $E$ by an angle $\pi/5$ viewed from $S'$.  
The group ${\rm Sym}(\mathcal{C})_0$ acts simply transitively on the set of flags of $\mathcal C$ of the form $(S, F, E)$, 
and so there is a unique element $\phi_F$ of ${\rm Sym}(\mathcal{C})_0$ such that $\phi_F(S,F,E) = (S',F,'E')$. 
The symmetry of $\phi_F$ rotates the dodecahedron $S$ to the adjacent dodecahedron $S'$ by a right-hand twist 
of $\pi/5$ radians as in Figure 10.1.1 of \cite{R} where $F$ is the front face.
The twelve symmetries $\{\phi_F: F\ \hbox{is a face of}\ S\}$ generate $A$ with $\phi_F^{-1} = \phi_{F'}$. 
Each symmetry $\phi_F$ has order 10 with $\phi_F^5 = -I$. 
The group $A$ acts simply transitively on the sides of $\mathcal C$, and so has order 120. 
The group $B$ is defined in the same way by taking left-hand twists. 

The groups $A$ and $B$ are mirror images of each other, in the sense that conjugating ${\rm Sym}(\mathcal C)_0$ by 
an orientation-reversing symmetry of $\mathcal C$ interchanges the subgroups $A$ and $B$. 
Each of the groups $A, B$ acts freely on $\partial \mathcal C$ with orbit space a Poincar\'e homology 3-sphere.

The groups $A, B$ lift isomorphically to subgroups $\hat A, \hat B$ of ${\rm Spin}^+(4,1)$, 
and we have $\widehat{\rm Sym}(\mathcal C)_0 =\hat A\times\hat B$. 
In terms of the representation ${\rm SU}(1,1;\mathbb{H})$ of ${\rm Spin}^+(4,1)$, the group
$\widehat{\rm Sym}(\mathcal C)_0$ is represented by the group of diagonal matrices ${\rm diag}(p,q)$ where 
$p$ and $q$ are in the binary icosahedral group $2I$ consisting of one of the 24 unit quaternions
$$\{\pm 1, \pm {\bf i}, \pm{\bf j}, \pm{\bf k}, (\pm 1\pm{\bf i} \pm{\bf j}\pm {\bf k})/2\}$$
or one of the 96 unit quaternions obtained from $(0\pm{\bf i}\pm \tau{\bf j} \pm \tau^{-1}{\bf k})/2$ 
by an even permutation of all the four coordinates $0,1,\tau,\tau^{-1}$. Note that $\tau^{-1} = \tau - 1$.
When convenient, we shall identify $\widehat{\rm Sym}(\mathcal C)_0 =\hat A\times \hat B$ with $2I\times 2I$. 

\begin{theorem}\label{T:4.1} 
The short exact sequence $1 \to \widehat{\rm Sym}(\mathcal C)_0 \to \hat G \to \langle \sigma_\star \rangle \to 1$ splits. 
The element $\hat\sigma_\star$ of $\hat G$ acts by conjugation on $\widehat{\rm Sym}(\mathcal C)_0 = 2I \times 2I$ 
so that $(p,q)$ maps to $(\alpha^{-1}(q),\alpha(p))$ 
where $\alpha$ is an outer automorphism of the binary icosahedral group $2I$ of order $4$ 
such that $\alpha = \alpha^2\alpha^{-1}$ with $\alpha^2$ an inner automorphism of $2I$. 
\end{theorem}
\begin{proof}
The group $A$ is generated by the following two elements of order 10.
$$\alpha_1 = \frac{1}{2}\left(\begin{array}{ccccc} 
\tau & -1 & 1-\tau & 0 & 0 \\
1 & \tau & 0 &-1+\tau & 0 \\
-1+\tau & 0 & \tau & -1 & 0 \\
0 & 1-\tau & 1 & \tau & 0 \\
0 & 0 & 0 & 0 & 2
\end{array}\right). $$
$$\alpha_2 = \frac{1}{2}\left(\begin{array}{ccccc} 
\tau & -1 & -1+\tau & 0 & 0 \\
1 & \tau & 0 &1-\tau & 0 \\
1-\tau & 0 & \tau & -1 & 0 \\
0 & -1+\tau & 1 & \tau & 0 \\
0 & 0 & 0 & 0 & 2
\end{array}\right).$$
Both $\alpha_1$ and $\alpha_2$ are right-hand twists, since they are conjugate in ${\rm SO}^+(4,1)$ to 
$$\left(\begin{array}{ccccc} 
\cos(\pi/5) & -\sin(\pi/5) & 0 & 0 & 0 \\
\sin(\pi/5) & \cos(\pi/5) & 0 &0 & 0 \\
0 & 0 & \cos(\pi/5) & -\sin(\pi/5) & 0 \\
0 & 0& \sin(\pi/5)& \cos(\pi/5) & 0 \\
0 & 0 & 0 & 0 & 1
\end{array}\right).$$
The elements $\alpha_1,\alpha_2$ lift to the following elements of ${\rm SU}(1,1;\mathbb{H})$, 
$$\hat\alpha_1 = \left(\begin{array}{cc} 
\frac{\tau}{2} + \frac{1}{2}{\bf i} + (-\frac{1}{2}+\frac{\tau}{2}){\bf j} & 0 \\
0 & 1 \end{array}\right), \quad
\hat\alpha_2 = \left(\begin{array}{cc} 
\frac{\tau}{2} + \frac{1}{2}{\bf i} + (\frac{1}{2}-\frac{\tau}{2}){\bf j} & 0 \\
0 & 1 \end{array}\right).$$
Therefore $\hat A$ is represented in ${\rm SU}(1,1;\mathbb{H})$ by $\{{\rm diag}(p,1): p \in 2I\}$. 

Conjugating $\alpha_1$ and $\alpha_2$ by ${\rm diag}(-1,1,1,1,1)$ gives two left-hand twists
$$\beta_1 = \frac{1}{2}\left(\begin{array}{ccccc} 
\tau & 1 & -1+\tau & 0 & 0 \\
-1 & \tau & 0 &-1+\tau & 0 \\
1-\tau & 0 & \tau & -1 & 0 \\
0 & 1-\tau & 1 & \tau & 0 \\
0 & 0 & 0 & 0 & 2
\end{array}\right), $$
$$\beta_2 = \frac{1}{2}\left(\begin{array}{ccccc} 
\tau & 1 & 1-\tau & 0 & 0 \\
-1 & \tau & 0 &1-\tau & 0 \\
-1+\tau & 0 & \tau & -1 & 0 \\
0 & -1+\tau & 1 & \tau & 0 \\
0 & 0 & 0 & 0 & 2
\end{array}\right).$$
of order 10 that generate the group $B$. 

The elements $\beta_1,\beta_2$ lift to the following elements of ${\rm SU}(1,1;\mathbb{H})$, 
$$\hat\beta_1 = \left(\begin{array}{cc} 
1 & 0 \\
0 & \frac{\tau}{2} + \frac{1}{2}{\bf i} + (-\frac{1}{2}+\frac{\tau}{2}){\bf j} \end{array}\right), \quad
\hat\beta_2 = \left(\begin{array}{cc} 
1 & 0 \\
0 & \frac{\tau}{2} + \frac{1}{2}{\bf i} + (\frac{1}{2}-\frac{\tau}{2}){\bf j}\end{array}\right).$$
Therefore $\hat B$ is represented in ${\rm SU}(1,1;\mathbb{H})$ by $\{{\rm diag}(1,q): q \in 2I\}$. 

The $(5,3,3,5)$ Coxeter group $W$ generated by the reflections in the sides of $\Delta$ is 
the group of symmetries of the regular tessellation of $H^4$ with cells $\{\gamma \mathcal C: \gamma \in \Gamma\}$.
The isometric involution $\sigma$ is a symmetry of $\Delta$, and so $\sigma$ normalizes $W$. 
As $\alpha_i \in W$ for each $i = 1, 2$, we have that $\sigma\alpha_i\sigma^{-1} \in W$ for each $i= 1,2$. 
Hence, there exists unique $\gamma_i \in \Gamma$ such that $\sigma\alpha_i\sigma^{-1}\mathcal C = \gamma_i \mathcal C$ 
for each $i = 1,2$. Therefore, $\gamma_i^{-1}\sigma\alpha_i\sigma^{-1} \in {\rm Sym}(\mathcal C)_0$ for each $i = 1,2$. 
We found that $\gamma_1^{-1} = g_{74}g_{36}g_{117}$ and $\gamma_2^{-1} = g_{93}g_{22}g_{117}$. 

Upon lifting to ${\rm SU}(1,1;\mathbb{H})$, we have that 
$$\hat\gamma_1^{-1}\hat\sigma\hat\alpha_1\hat\sigma^{-1} =  \left(\begin{array}{cc} 
1 & 0 \\
0 & \frac{1}{2} - \frac{\tau}{2} +\frac{\tau}{2}{\bf i} + \frac{1}{2}{\bf j} \end{array}\right) = \hat\beta_1^3,$$
$$\hat\gamma_2^{-1}\hat\sigma\hat\alpha_2\hat\sigma^{-1} =  \left(\begin{array}{cc} 
1 & 0 \\
0 & \frac{1}{2} - \frac{\tau}{2} -\frac{\tau}{2}{\bf i} + \frac{1}{2}{\bf j} \end{array}\right) = \hat\beta_2^7.$$
Therefore, the action of $\sigma_\star$ on $\widehat{\rm Sym}(\mathcal C)_0$ induced by conjugation in $\hat G$ 
maps $\hat A$ onto $\hat B$, and vice versa, since conjugation by $\hat\sigma_\star$ is an involution  
since $\hat\sigma^2 = -1$. 

The corresponding automorphism $\alpha$ of $2I$, defined by 
\begin{eqnarray*}
\textstyle{\alpha\big(\frac{\tau}{2} + \frac{1}{2}{\bf i} + (-\frac{1}{2}+\frac{\tau}{2}){\bf j}\big)}\!\! & =  & \textstyle{\frac{1}{2} - \frac{\tau}{2} +\frac{\tau}{2}{\bf i} + \frac{1}{2}{\bf j}} \ \, = \ \, \textstyle{\big(\frac{\tau}{2} + \frac{1}{2}{\bf i} + (-\frac{1}{2}+\frac{\tau}{2}){\bf j}\big)^3} \\
\textstyle{\alpha\big(\frac{\tau}{2} + \frac{1}{2}{\bf i} + (\frac{1}{2}-\frac{\tau}{2}){\bf j}\big)}\!\!& =  &\!\! \textstyle{\frac{1}{2} - \frac{\tau}{2} -\frac{\tau}{2}{\bf i} + \frac{1}{2}{\bf j}}\ \, = \ \, \textstyle{\big(\frac{\tau}{2} + \frac{1}{2}{\bf i} + (\frac{1}{2}-\frac{\tau}{2}){\bf j}\big)^7},\end{eqnarray*}
is an outer automorphism, since if $p$ and $q$ are quaternions, with $q\neq 0$, then $p$ and $qpq^{-1}$ 
have the same real part. Observe that $\alpha$ has order 4. The outer automorphism group of $2I$ has order 2, 
and so $\alpha^2$ is an inner automorphism. In fact, $\alpha^2(p) = {\bf k}p{\bf k}^{-1}$. 
Note that $\alpha(-1) = -1$, since $\langle -1\rangle$ is the center of $2I$.  

In the decomposition $\widehat{\rm Sym}(\mathcal C)_0 = 2I \times 2I$, 
the center of the first factor is $\langle(-1,1)\rangle$ and the center of the second factor is $\langle(1,-1)\rangle$. 
The element $(1,-1)_\star\hat\sigma_\star$ of $\hat G$ has order 2, 
since 
$$(1,-1)_\star\hat\sigma_\star(1,-1)_\star \hat\sigma_\star = (1,-1)_\star (-1,1)_\star \hat\sigma^2_\star = 1.$$
Finally, the exact sequence  $1 \to \widehat{\rm Sym}(\mathcal C)_0 \to \hat G \to \langle \sigma_\star \rangle \to 1$ splits, 
since $(1,-1)_\star\hat\sigma_\star$ maps to $\sigma_\star$ in the sequence. 
\end{proof}

\section{The Representation Theory of the Group $\hat G$}\label{S:5}  

In this section, we describe the complex representation theory of the the group $\hat G$ of order 28,800. 
Our goal is to find all the irreducible representations $\rho$ of $\hat G$ that are {\it spinorial} in the sense 
that $\rho(-1) = -\rho(1)$. 

By Theorem \ref{T:4.1}, we have a split, short, exact sequence
$$1 \to \widehat{\rm Sym}(\mathcal C)_0 \to \hat G \to \langle \sigma_\star \rangle \to 1.$$
The group $\widehat{\rm Sym}(\mathcal C)_0$ is a direct product $\hat A \times \hat B$ 
with each factor isomorphic to the binary icosahedral group $2I$ of order 120.

We begin with the complex representation theory of $2I$. 
The group $2I$ consists of the 24 unit quaternions
$$\{\pm 1, \pm {\bf i}, \pm{\bf j}, \pm{\bf k}, (\pm 1\pm{\bf i} \pm{\bf j}\pm {\bf k})/2\}$$
together with the 96 unit quaternions obtained from $(0\pm{\bf i}\pm \tau{\bf j} \pm \tau^{-1}{\bf k})/2$ 
by an even permutation of all the four coordinates $0,1,\tau,\tau^{-1}$. Note that $\tau^{-1} = \tau - 1$. 

The group $2I$ has exactly 9 conjugacy classes, {\bf 1, 2, 3, 4, 5A, 5B, 6, 10A, 10B}, 
with the numerical part of the name of a conjugacy class the order of each element of the class. 
In Table \ref{ta:2}, the conjugacy classes of $2I$ are classified by the constant real part of each quaternion in the class. 
Hence, $2I$ has exactly 9 irreducible representations, ${\bf 1, 2, 2', 3, 3', 4, 4', 5, 6}$,   
with the numerical part of the name of a representation the complex dimension of the representation. 
A character table for $2I$ is given in Table \ref{ta:3}. For the irreducible representations of $2I$, see \cite{C-P}.

\begin{table} 
$$\begin{array}{|c|rrccccccc|} \hline
Class & {\bf 1} & {\bf 2} & {\bf 3} & {\bf 4} & {\bf 5A} & {\bf 5B} & {\bf 6} &  {\bf 10A} & {\bf 10B} \\ \hline
Re(q) & 1 & -1 & -1/2 & 0 &  (\tau-1)/2 & -\tau/2 & 1/2 & \tau/2 & (1-\tau)/2 \\ 
Size & 1 & 1 & 20 & 30 & 12 & 12 & 20 & 12 & 12 \\ \hline
\end{array}$$
\caption{The conjugacy classes of the binary icosahedral group $2I$}\label{ta:2}
\end{table}

\begin{table} 
$$\begin{array}{|c|rrrrrrrrr|} \hline
Rep & {\bf 1} & {\bf 2} & {\bf 3} & {\bf 4} & {\bf 5A} & {\bf 5B} & {\bf 6} &  {\bf 10A} & {\bf 10B} \\ \hline
{\bf 1} & 1 & 1 & 1 & 1 & 1 & 1 & 1 & 1 & 1\\ 
{\bf 2} & 2 & -2 & -1& 0 & \tau-1 & -\tau & 1 & \tau & 1-\tau\\ 
{\hspace{.04in}\bf 2'} & 2 & -2 & -1& 0 & -\tau & \tau-1 & 1 & 1-\tau & \tau\\ 
{\bf 3} & 3 & 3 & 0 & -1 & 1-\tau & \tau & 0 & \tau & 1-\tau\\ 
{\hspace{.04in}\bf 3'} & 3 & 3 & 0 & -1 & \tau & 1-\tau & 0 & 1-\tau & \tau\\ 
{\bf 4} & 4 & 4 & 1 & 0 & -1 & -1 & 1 & -1 & -1\\ 
{\hspace{.04in}\bf 4'} & 4 & -4 & 1 & 0 & -1 & -1 & -1 & 1 & 1\\
{\bf 5} & 5 & 5 & -1 & 1 & 0 & 0 & -1 & 0 & 0\\ 
{\bf 6} & 6 & -6 & 0 & 0 & 1 & 1 & 0 & -1 & -1 \\ \hline
\end{array}$$
\caption{The character table for the binary icosahedral group $2I$}\label{ta:3}
\end{table}

A representation $\rho$ of $2I$ is said to be {\it spinorial} if $\rho(-1) = -\rho(1)$. 
From the second column of Table \ref{ta:3}, we see that the spinorial irreducible representations of $2I$ are ${\bf 2,2',4', 6}$. 

Every irreducible representation of $2I \times 2I$ is of the form $\rho_1\otimes \rho_2$ where $\rho_1$ and $\rho_2$ 
are irreducible representations of $2I$.  Therefore $2I \times 2I$ has $81$ irreducible representations. 
A representation $\rho_1\otimes \rho_2$ of $2I \times 2I$ is said to be {\it spinorial} if $$(\rho_1\otimes \rho_2)(-1,-1) = -(\rho_1(1)\otimes \rho_2(1)).$$  As $(\rho_1\otimes \rho_2)(-1,-1) = \rho_1(-1)\otimes\rho_2(-1)$, this will be the case if and only if one of $\rho_1, \rho_2$ is spinorial and the other is not. Hence, there are exactly 40 spinorial irreducible representations of $2I\times 2I$. 
As $\dim (\rho_1\otimes \rho_2) = (\dim\rho_1)(\dim\rho_2)$, all the spinorial irreducible representations of $2I\times 2I$ are even dimensional.

We have a split, short, exact sequence 
$$1\to 2I\times 2I \to \hat G \to \langle \sigma_\star \rangle \to 1.$$
Let $s$ be an element of $\hat G$ of order 2 that maps to $\sigma_\star$ in the above sequence. 
Let $\theta$ be a representation of $2I \times 2I$, and let $\theta^s$ be the representation of $2I \times 2I$ defined by 
$$\theta^s(p,q) = \theta(s^{-1}(p,q)s) = \theta(\alpha^{-1}(q), \alpha(p))$$
where $\alpha$ is the outer automorphism of $2I$ of order 4 from Theorem \ref{T:4.1} such that $\alpha = \alpha^2\alpha^{-1}$ and 
$\alpha^2$ is an inner automorphism of $2I$.

The representation of $\hat G$ {\it induced} by $\theta$
is the representation $\rho$ of $\hat G$ that extends the representation $\theta\oplus \theta^s$ of $2I\times 2I$ 
and such that $s$ acts by $s(v\oplus w) = w\oplus v$. 
It follows from Mackey's irreduciblity criterion that 
all the spinorial irreducible representations of $\hat G$ are induced by spinorial irreducible 
representations of $2I\times 2I$. 

Let $\rho_1$ and $\rho_2$ be representations of $2I$ such that $\rho_1\otimes \rho_2$ is a spinorial representation of $2I\times 2I$. 
Then 
$$(\rho_1\otimes \rho_2)^s (p,q) = (\rho_1\otimes\rho_2)(\alpha^{-1}(q),\alpha(p)) = \rho_1(\alpha^{-1}(q))\otimes\rho_2(\alpha(p)).$$
This implies that 
$$(\rho_1\otimes \rho_2)^s \cong \rho_2' \otimes \rho_1'$$
where $\rho_i' = \rho_i$ if $\dim\rho_i \neq 2,3$ or else 
$\rho_i'$ is the other irreducible representation of the dimension of $\rho_i$ when $\dim\rho_i = 2,3$. 
Thus, the corresponding induced representation of $\hat G$ is isomorphic to 
a representation that extends $(\rho_1\otimes\rho_2)\oplus(\rho_2'\otimes\rho_1')$.

The group $\hat G$ has exactly 54 conjugacy classes, and so $\hat G$ has exactly 54 irreducible representations 
up to isomorphism. There are exactly 20 spinorial irreducible representation of $\hat G$. 
These representations extend the reducible representations of $2I\times 2I$ listed in Table \ref{ta:4}. 

\begin{table} 
$$\begin{array}{|c|l|c|l|} \hline
Dim & Representation & Dim & Representation \\ \hline
4 & ({\bf 1}\otimes {\bf 2})\oplus({\bf 2'}\otimes {\bf 1}) & 20 & ({\bf 2}\otimes {\bf 5})\oplus({\bf 5}\otimes {\bf 2'}) \\ 
4 & ({\bf 1}\otimes {\bf 2'})\oplus({\bf 2}\otimes {\bf 1}) & 20 & ({\bf 2'}\otimes {\bf 5})\oplus({\bf 5}\otimes {\bf 2}) \\ 
8 & ({\bf 1}\otimes {\bf 4'})\oplus({\bf 4'}\otimes {\bf 1}) & 24 & ({\bf 3}\otimes {\bf 4'})\oplus({\bf 4'}\otimes {\bf 3'})  \\ 
12 & ({\bf 1}\otimes {\bf 6})\oplus({\bf 6}\otimes {\bf 1}) & 24 & ({\bf 3'}\otimes {\bf 4'})\oplus({\bf 4'}\otimes {\bf 3})  \\ 
12 & ({\bf 2}\otimes {\bf 3})\oplus({\bf 3'}\otimes {\bf 2'}) & 32 & ({\bf 4}\otimes {\bf 4'})\oplus({\bf 4'}\otimes {\bf 4})  \\ 
12 & ({\bf 2}\otimes {\bf 3'})\oplus({\bf 3}\otimes {\bf 2'}) & 36 & ({\bf 3}\otimes {\bf 6})\oplus({\bf 6}\otimes {\bf 3'})  \\ 
12 & ({\bf 2'}\otimes {\bf 3})\oplus({\bf 3'}\otimes {\bf 2}) & 36 & ({\bf 3'}\otimes {\bf 6})\oplus({\bf 6}\otimes {\bf 3})  \\ 
12 & ({\bf 2'}\otimes {\bf 3'})\oplus({\bf 3}\otimes {\bf 2}) & 40 & ({\bf 4'}\otimes {\bf 5})\oplus({\bf 5}\otimes {\bf 4'}) \\ 
16 & ({\bf 2}\otimes {\bf 4})\oplus({\bf 4}\otimes {\bf 2'}) & 48 & ({\bf 4}\otimes {\bf 6})\oplus({\bf 6}\otimes {\bf 4})  \\ 
16 & ({\bf 2'}\otimes {\bf 4})\oplus({\bf 4}\otimes {\bf 2}) & 60 & ({\bf 5}\otimes {\bf 6})\oplus({\bf 6}\otimes {\bf 5})  \\ \hline
\end{array}$$

\vspace{.15in}
\caption{The spinorial irreducible representations of $\hat G$}\label{ta:4}
\end{table}

The group $G$ of orientation-preserving isometries of the Davis hyperbolic 4-manifold has exactly 34 conjugacy classes, 
and so $G$ has exactly 34 irreducible representations. 
The 34 nonspinorial irreducible representations of $\hat G$ are lifts of the 34 irreducible representations of $G$. 

The nonspinorial irreducible representations of $\hat G$ are of three types; they are either induced, and so are of the same form as the spinorial irreducible representations, or they are of the form $\pm (\rho_1\otimes \rho_2)$ with $\rho_1$ and $\rho_2$ 
irreducible representations of $2I$ and $\pm (\rho_1\otimes \rho_2)$ extending $\rho_1\otimes \rho_2$, 
and $s$ acting by $s(v\otimes w) = \pm (w\otimes v)$.  
The 34 nonspinorial irreducible representations of $\hat G$ are listed in Table \ref{ta:5}. 

We checked that we have found all the irreducible representations of $\hat G$ by computing a character table for $\hat G$ 
and verifying all the orthonormality conditions for a character table. 
All the entries of the character table are algebraic integers in $\mathbb Q [\tau]$. 

\begin{table} 
$$\begin{array}{|c|l|c|l|} \hline
Dim & Representation & Dim & Representation \\ \hline
1 & \phantom{-1}{\bf 1}\otimes {\bf 1} & 16 &  \phantom{-1}{\bf 4}\otimes {\bf 4}\ \\ 
1 & -({\bf 1}\otimes {\bf 1})                  & 16 & -( {\bf 4}\otimes {\bf 4}) \\ 
4 & \phantom{-1}{\bf 2}\otimes {\bf 2'} & 16 &  \phantom{-1}{\bf 4'}\otimes {\bf 4'} \\ 
4 & -({\bf 2}\otimes {\bf 2'})                  & 16 & -({\bf 4'}\otimes {\bf 4'}) \\ 
4 & \phantom{-1}{\bf 2'}\otimes {\bf 2} & 18 & ({\bf 3}\otimes {\bf 3})\oplus({\bf 3'}\otimes {\bf 3'}) \\ 
4 & -({\bf 2'}\otimes {\bf 2})                  & 24 & ({\bf 2}\otimes {\bf 6})\oplus({\bf 6}\otimes {\bf 2'}) \\ 
6 & ({\bf 1}\otimes {\bf 3})\oplus({\bf 3'}\otimes {\bf 1}) & 24 & ({\bf 2'}\otimes {\bf 6})\oplus({\bf 6}\otimes {\bf 2})  \\ 
6 & ({\bf 1}\otimes {\bf 3'})\oplus({\bf 3}\otimes {\bf 1}) & 24 & ({\bf 3}\otimes {\bf 4})\oplus({\bf 4}\otimes {\bf 3'})  \\ 
8 & ({\bf 1}\otimes {\bf 4})\oplus({\bf 4}\otimes {\bf 1}) & 24 & ({\bf 3'}\otimes {\bf 4})\oplus({\bf 4}\otimes {\bf 3})  \\ 
8 & ({\bf 2}\otimes {\bf 2})\oplus({\bf 2'}\otimes {\bf 2'}) & 25 & \phantom{-1}{\bf 5}\otimes {\bf 5}  \\ 
9 & \phantom{-1}{\bf 3}\otimes {\bf 3'}                           & 25 & -({\bf 5}\otimes {\bf 5}) \\ 
9 & -({\bf 3}\otimes {\bf 3'})& 30 & ({\bf 3}\otimes {\bf 5})\oplus({\bf 5}\otimes {\bf 3'}) \\ 
9 & \phantom{-1}{\bf 3'}\otimes {\bf 3} & 30 & ({\bf 3'}\otimes {\bf 5})\oplus({\bf 5}\otimes {\bf 3}) \\ 
9 & -({\bf 3'}\otimes {\bf 3})                                             & 36 &\phantom{-1}{\bf 6}\otimes {\bf 6}  \\ 
10 & ({\bf 1}\otimes {\bf 5})\oplus({\bf 5}\otimes {\bf 1}) & 36 &-({\bf 6}\otimes {\bf 6})   \\ 
16 & ({\bf 2}\otimes {\bf 4'})\oplus({\bf 4'}\otimes {\bf 2'}) & 40 & ({\bf 4}\otimes {\bf 5})\oplus({\bf 5}\otimes {\bf 4}) \\ 
16 & ({\bf 2'}\otimes {\bf 4'})\oplus({\bf 4'}\otimes {\bf 2}) & 48 & ({\bf 4'}\otimes {\bf 6})\oplus({\bf 6}\otimes {\bf 4'})  \\ \hline
\end{array}$$

\vspace{.15in}
\caption{The nonspinorial irreducible representations of $\hat G$}\label{ta:5}
\end{table}

\section{On totally geodesic submanifolds in a hyperbolic manifold}\label{S:6}  

The next theorem is known to experts, cf. p 38 \cite{G-L-T}. 

\begin{theorem}\label{T:6.1} 
If $\Sigma$ is a totally geodesic, orientable, embedded, closed submanifold, of codimension $k >0$, of a 
connected, orientable, closed, hyperbolic manifold $M$,   
then 
\begin{enumerate}
\item The normal bundle $\nu(\Sigma)$ is flat. 
\item The normal bundle $\nu(\Sigma)$ is orientable. 
\item If $k$ is even, then the Euler class of $\nu(\Sigma)$ is zero. 
\item If $k = 2$, then $\nu(\Sigma)$ is trivial. 
\end{enumerate}
\end{theorem}
\begin{proof}
(1) The second fundamental form $II$ of $\Sigma$ is zero, since $\Sigma$ is totally geodesic, cf. p 33 \cite{K}. 
Hence, the shape operator $A_v$ of $\Sigma$ in any normal direction $v$ is zero, since 
$\langle II(u_1,u_2),v\rangle = \langle A_v(u_1),u_2\rangle$, cf. p 26 \cite{K}. 
Therefore, the normal curvature $\Omega^\nu$ of $\Sigma$ is zero, since 
$\Omega^\nu(u,v) = [A_u,A_v]$ by Proposition 2.1.1, p 29 \cite{K}. 
Therefore $\nu(\Sigma)$ is a flat bundle, cf. p 29 \cite{K}. 

(2) We have that $T(\Sigma)\oplus \nu(\Sigma) \cong T(M)|_\Sigma$ by Corollary 3.4 of \cite{M-S}.
A smooth manifold is orientable if and only if its tangent bundle is orientable by Lemma 11.6 of \cite{M-S}. 
Hence $T(\Sigma)$ and $T(M)$ are orientable, and so $T(M)|_\Sigma$ is orientable. 
A vector bundle over a smooth manifold is orientable if and only if its first Stiefel-Whitney class is zero, cf. p 148 of \cite{M-S}. 
By the Whitney Product Theorem, pp 37-38 \cite{M-S}, we have that 
$$w_1(\nu(\Sigma))) = w_1(T(\Sigma)) + w_1(\nu(\Sigma)) = w_1(T(M)|_\Sigma) = 0.$$
Therefore $\nu(\Sigma)$ is orientable.   

(3) Assume that $k$ is even. The normal connection $\nabla^\nu$ is compatible with the Euclidean metric on $\nu(\Sigma)$ induced by the Riemannian metric on $M$ by Lemma 7 on p 300 \cite{M-S}, the second Formula 2.1.1 on p 26 \cite{P-T}, and the last formula on p 27 \cite{P-T}. 
Therefore ${\rm Pf}(\Omega^\nu/2\pi)$ represents the Euler class $e(\nu(\Sigma))$ 
by the Generalized Gauss-Bonnet Theorem, p 311 \cite{M-S}. 
Hence $e(\nu(\Sigma)) = 0$, since $\Omega^{\nu} = 0$ by Part 1.

(4) Assume that $k =2$. The normal bundle $\nu(\Sigma)$ is oriented by the orientations of $\Sigma$ and $M$ by Part 2, 
and so $\nu(\Sigma)$ has a canonical complex structue, cf. p 305 \cite{M-S}. 
Now $c_1(\nu(\Sigma)) = e(\nu(\Sigma))$, and so $c_1(\nu(\Sigma)) = 0$ by Part 3. 
Therefore $\nu(\Sigma)$ is trivial, since the first Chern class is a complete invariant of complex line bundles. 
\end{proof}

\begin{theorem}\label{T:6.2} 
If $\phi$ is an orientation-preserving isometry of a connected, orientable, closed, Riemannian $2m$-manifold $M$ 
that fixes a point of $M$, then  
\begin{enumerate} 
\item The set $M^\phi$ of fixed points of $\phi$ has only finitely many connected components, 
and each connected component of $M^\phi$ is a totally geodesic, closed, even-dimensional,  embedded submanifold of $M$. 
\item If the order of $\phi$ is $2$ and $\phi$ lifts to a symmetry of order $4$ of a spin structure on $M$, 
then the co-dimension of each connected component of $M^\phi$ is congruent to $2$ modulo $4$. 
If $m$ is even, each fixed point of $\phi$ is non-isolated. 
\item If the order of $\phi$ is $2$ and $\phi$ lifts to a symmetry of order $2$ of a spin structure on $M$, 
then the co-dimension of each connected component of $M^\phi$ is congruent to $0$ modulo $4$.  
If $m$ is odd, each fixed point of $\phi$ is non-isolated. 
\item If the order of $\phi$ is even and $\phi$ lifts to a symmetry, of the same order, of a spin structure on $M$, and $m =2$, 
then each fixed point of $\phi$ is isolated.  
\item If the order of $\phi$ is greater than $2$, then each codimension-two connected component of $M^\phi$ is orientable.
\end{enumerate}
\end{theorem}
\begin{proof}
(1) Let $C$ be a connected component of $M^\phi$. Then $C$ is a totally geodesic submanifold of $M$ by Theorem II.5.1 of \cite{K}, 
whose proof shows that $C$ is an embedded submanifold of $M$ and $C$ is an open subset of $M^\phi$. 
The set $M^\phi$ is compact, since $M$ is compact and $M^\phi$ is a closed subset of $M$. 
Therefore $M^\phi$ has only finitely many connected components. 
A component $C$ of $M^\phi$ is a closed manifold, since $C$ is a closed subset of $M^\phi$, 
and so $C$ is compact. The dimension of $C$ is the same as the multiplicity of the $+1$ eigenvalues of the action of $\phi$ 
on a fiber of the tangent bundle of $M$ over a point of $C$, and so ${\rm dim}(C)$ is even, since ${\rm dim}(M)$ is even.

(2) Assume that ${\rm ord}(\phi) = 2$, and $\phi$ lifts to a symmetry $\hat\phi$ of order 4 of a spin structure on $M$. 
Then the co-dimension of each connected component of $M^\phi$ is congruent to $2$ modulo $4$ by Proposition 8.46 of \cite{A-B}. 
If $m$ is even, then ${\rm dim}(M) \equiv 0\ \hbox{mod}\ 4$, and so the dimension of each connected component of $M^\phi$ is positive, 
and so each fixed point of $\phi$ is non-isolated. 

(3) Assume that ${\rm ord}(\phi) = 2$, and $\phi$ lifts to a symmetry $\hat\phi$ of order 2 of a spin structure on $M$. 
Then the co-dimension of each connected component of $M^\phi$ is congruent to $0$ modulo $4$ by Proposition 8.46 of \cite{A-B}. 
If $m$ is odd, then ${\rm dim}(M) \equiv 2\ \hbox{mod}\ 4$, and so the dimension of each connected component of $M^\phi$ is positive, 
and so each fixed point of $\phi$ is non-isolated. 

(4) Assume that ${\rm ord}(\phi) = 2k$, with $k$ a positive integer, 
$\phi$ lifts to a symmetry $\hat\phi$, of order $2k$, of a spin structure on $M$, and $m = 2$. 
Then $\phi^k$ and $\hat\phi^k$ have order 2.  
As $2m =4$, we have that $\phi^k$ has only finitely many fixed points by Parts 1 and 3. 
As every point fixed by $\phi$ is also fixed by $\phi^k$, 
every fixed point of $\phi$ is isolated.

(5) Assume ${\rm ord}(\phi) > 2$. Let $C$ be a co-dimension-two connected component of $M^\phi$, let $\nu(C)$ be the normal bundle of $C$, and 
let $v$ be a normal vector in a fiber of $\nu(C)$ over a point $x$ of $C$. Then the ordered pair $(v, \phi_\ast v)$ is a basis 
for the fiber of $\nu(C)$ over $x$ that is independent of the choice of $v$, since $\phi_\ast$ is a rotation by the angle $2\pi/{\rm ord}(\phi)$. 
Therefore, $\nu(C)$ is oriented by the induced action of $\phi$ on each fiber of $\nu(C)$.  
By the argument in the proof of Theorem 6.1(2), we have that 
$$w_1(T(C)) = w_1(T(C)) + w_1(\nu(C)) = w_1(T(M)|_C) = 0.$$
Therefore $T(C)$ is orientable, and so $C$ is orientable.
\end{proof}

\section{Spin Numbers}\label{S:7}  

Let $\phi$ be an orientation-preserving isometry of a spin, closed, hyperbolic $2m$-manifold $M = \Gamma\backslash H^{2m}$, with $\Gamma$ a discrete torsion-free subgroup of ${\rm SO}^+(2m,1)$. 
Suppose that $\phi$ lifts to a symmetry $\hat \phi$ of the spin structure $\hat\Gamma\backslash{\rm Spin}^+(2m,1)$ of $M$, 
where $\hat\Gamma$ is a discrete subgroup of ${\rm Spin}^+(2m,1)$ that projects isometrically onto $\Gamma$.  
Then $\hat \phi$ acts on the spaces ${\mathcal H}^{\pm}$ of positive (negative) harmonic spinors of the spin structure 
giving two spinorial representations $\rho^{\pm}: \hat G \to {\rm GL}(\mathcal H^{\pm})$. 

The {\it spin number} of $\hat \phi$ is defined to be 
\begin{equation}\label{E:1}
{\rm Spin}(\hat \phi, M) = {\rm tr}(\rho^+(\hat \phi)) - {\rm tr}(\rho^-(\hat \phi)).
\end{equation}
Note that the $\hat G$-index, ${\rm Spin}(\hat G,M)$, is determined by the set of spin numbers $\{{\rm Spin}(\hat \phi,M):\hat\phi \in \hat G\}$, 
since a representation is determined by its character.  

From Formula \ref{E:1}, we see that ${\rm Spin}(\hat \phi, M)$ is an algebraic integer, since $\hat\phi$ has finite order. 
For example, if $\phi = 1$, then $\hat \phi = \pm 1$, and ${\rm Spin}(\hat \phi, M) = 0$, 
since ${\rm dim}\, {\mathcal H}^+ = {\rm dim}\, {\mathcal H}^-$. 
The spin number ${\rm Spin}(\hat \phi, M)$ may depend on the lift $\hat \phi$ of $\phi$, of 
which there are two $\pm\hat \phi$, since 
\begin{equation}\label{E:2}
{\rm Spin}(-\hat \phi, M) = -{\rm Spin}(\hat \phi, M).
\end{equation}
Spin numbers are obviously invariant under conjugation, that is, 
if $\hat \psi$ is a lift of an orientation-preserving isometry $\psi$ of $M$ to a symmetry of the spin structure, 
then 
\begin{equation}\label{E:3}
{\rm Spin}(\hat \psi\hat \phi\hat \psi^{-1}, M) = {\rm Spin}(\hat \phi, M).
\end{equation}

\begin{lemma}\label{L:7.1} 
If $\hat \psi\hat \phi\hat \psi^{-1} = -\hat \phi$, then $\phi$ has even order, and ${\rm Spin}(\hat \phi,M) = 0$. 
\end{lemma}
\begin{proof}
As $\hat \psi\hat \phi\hat \psi^{-1} = -\hat \phi$, the lifts $\pm \hat \phi$ have the same order. 
On the contrary, suppose that the order $k$ of $\phi$ is odd. 
Then $\hat \phi^k = \pm 1$. 
If $\hat \phi^k = 1$, then $\hat \phi$ has order $k$ and $-\hat \phi$ has order $2k$. 
If $\hat \phi^k = -1$, then $\hat \phi$ has order $2k$ and $-\hat \phi$ has order $k$. 
Therefore, the order of $\phi$ must be even. Moreover, by Formulas \ref{E:2} and \ref{E:3}, we have
$${\rm Spin}(\hat\phi, M) = {\rm Spin}(\hat \psi\hat \phi\hat \psi^{-1}, M) = {\rm Spin}(-\hat \phi, M) = -{\rm Spin}(\hat \phi,M).$$

\vspace{-.23in}
\end{proof}

By Theorem \ref{T:6.2}(1), the set $M^\phi$ of fixed points of $\phi$ has only finitely 
many connected components, and if $M^\phi$ is nonempty, then each component is a totally geodesic, closed, 
even-dimensional, embedded submanifold of $M$. 

By the Atiyah-Singer $G$-spin theorem,
\begin{equation}\label{E:4}
{\rm Spin}(\hat \phi, M) = \sum \nu(\hat \phi, C)
\end{equation}
where $C$ varies over the connected components of the set $M^\phi$ of fixed points of $\phi$, 
and $\nu(\hat \phi, C)$ is defined in terms of characteristic classes, cf. p 174 \cite{S} or \cite{L-M}. 

\begin{lemma}\label{L:7.2} 
Let $\phi$ be an orientation-preserving isometry of a connected, spin, closed, hyperbolic $4$-manifold $M$, 
and let $\hat \phi$ be a lift of $\phi$ to a symmetry of the spin structure on $M$. 
Then $\nu(\hat \phi, C) = 0$ for each surface component $C$ of $M^\phi$. 
\end{lemma}
\begin{proof}
Let $C$ be a surface component of $M^\phi$. By the $G$-spin theorem, p 174 of \cite{S}, 
$$\nu(\hat \phi, C)= \pm (\hat{\mathcal A}(C)\mathcal F(N))[C]$$
where $\hat{\mathcal A}(C)$ is a polynomial in the Pontryagin classes of the tangent bundle of $C$, 
which all live in $H^{4j}(C)$, and so $\hat{\mathcal A}(C)$ is a constant polynomial.  

If ${\rm ord}(\phi) = 2$, then $\mathcal F(N)$ is a polynomial in the Pontryagin classes of the normal bundle $N$ of $C$ in $M$, which all live in $H^{4j}(C)$, 
and so  $\mathcal F(N)$ is a constant polynomial. 
If ${\rm ord}(\phi) > 2$, then $C$ is orientable by Theorem \ref{T:6.2}(5), and so $N$ has a complex structure, 
and $\mathcal F(N)$ is a polynomial in the Chern classes of $N$, 
and so $\mathcal F(N)$ is a constant polynomial, since $N$ is trivial by Theorem \ref{T:6.1}(4). 

The total cohomology class $\hat{\mathcal A}(C)\mathcal F(N)$, 
with only one nonzero term in $H^0(C)$,
is evaluated on the fundamental class $[C]$ in $H_2(C)$, and so $\nu(\hat \phi, C)=0$. 
\end{proof}

Assume that $\phi$ has an isolated fixed point $P$. 
We now work toward deriving a new, useful, general formula for computing $\nu(\hat \phi, P)$ when $\dim M = 4$. 

Let $n = 2m$, and let $\eta:{\rm Spin}^+(n,1) \to {\rm SO}^+(n,1)$ be the double-covering epimorphism. 
The subgroup $\hat\Gamma$ of ${\rm Spin}^+(n,1)$ is mapped isomorphically onto $\Gamma$ by $\eta$. 
The map $\varpi: \hat\Gamma\backslash{\rm Spin}^+(n,1) \to \Gamma\backslash H^n$, defined by 
$\varpi(\hat\Gamma\hat g) = \Gamma(\eta(\hat g)e_{n+1})$  
is a principle ${\rm Spin}(n)$-bundle with right action induced by right multiplication in ${\rm Spin}^+(n,1)$. 
Here we identify ${\rm SO}(n)$ with the stabilizer of $e_{n+1}$ in ${\rm SO}^+(n,1)$, 
and we identity ${\rm Spin}(n)$ with $\eta^{-1}({\rm SO}(n))$. 

Let $g$ be an element of ${\rm SO}^+(n,1)$ such that $\Gamma ge_{n+1} = P$. 
Lift $g$ to an element $\hat g$ of ${\rm Spin}^+(n,1)$, which is unique up to multiplication by $-1$. 
Then $\varpi(\hat\Gamma\hat g) = P$.  

Let $f$ be an element of ${\rm SO}^+(n,1)$ such that $\phi = f_\star$, and let $\hat f$ in ${\rm Spin}^+(n,1)$ be the lift of $f$ 
such that $\hat \phi = \hat f_\star$. Then
$$\varpi(\hat\Gamma\hat f\hat g) = \Gamma(\eta(\hat f\hat g)e_{n+1}) = \Gamma (fge_{n+1}) = \phi(\Gamma ge_{n+1}) = \phi(P) = P. $$
Hence, there exists a unique element $s(\hat\phi,P;\hat\Gamma\hat g) = s$ in ${\rm Spin}(n)$ 
such that $\hat\Gamma \hat f \hat g = \hat\Gamma \hat g s$. 
Moreover $s$ does not depend on the choice of $\hat g$, 
since we can multiply the equation $\hat\Gamma \hat f \hat g = \hat\Gamma \hat g s$ by the central element $-1$ of 
${\rm Spin}^+(n,1)$. However $s$ does depend on the choice of the lift $\hat\phi$, since $\hat\Gamma(-\hat f)\hat g = \hat\Gamma\hat g(-s)$ implies that 
\begin{equation}\label{E:5}
s(-\hat\phi,P;\hat\Gamma\hat g) = -s(\hat\phi,P;\hat\Gamma\hat g).
\end{equation}

If $s_1 \in {\rm Spin}(n)$, then 
$$\hat\Gamma\hat f\hat g s_1 = \hat\Gamma\hat g s s_1 = \hat\Gamma \hat g s_1 s_1^{-1} s s_1,$$
and so the conjugacy class of $s(\hat\phi,P;\hat\Gamma\hat g)$ in ${\rm Spin}(n)$ does not depend 
on the choice of $\hat\Gamma\hat g$ in $\varpi^{-1}(P)$. 

By the $G$-Spin theorem (p 174 of \cite{S}), Formula 8.37 of \cite{A-B} and its proof are true in general 
for an isolated fixed point $P$.  
By Formulas 8.38 - 8.40 of \cite{A-B}, we have 
\begin{equation}\label{E:6}
\nu(\hat \phi,P) = \frac{{\rm tr}(\Delta_n^+(s(\hat\phi,P;\hat\Gamma\hat g))) - {\rm tr}(\Delta_n^-(s(\hat\phi,P;\hat\Gamma\hat g)))}{|{\rm det}(I-d\phi_P)|}
\end{equation}
where $\Delta^+_n$ and $\Delta^-_n$ are the positive and negative complex spin representations of ${\rm Spin}(n)$. 
Note that $\nu(\hat \phi,P)$ is a nonzero algebraic number by Formula 8.37 of \cite{A-B}.

It follows from Formulas \ref{E:5} and \ref{E:6} that 
\begin{equation}\label{E:7}
\nu(-\hat\phi,P) = -\nu(\hat\phi,P).
\end{equation}

Let $\psi$ be an orientation-preserving isometry of $M$ that lifts to a symmetry $\hat\psi$ of the spin structure on $M$. 
Then $\psi(P)$ is an isolated fixed point of $\psi\phi\psi^{-1}$. 
Let $h$ be an element of ${\rm SO}^+(n,1)$ such that $\psi = h_\star$, and let $\hat h$ in ${\rm Spin}^+(4,1)$ 
be the lift of $h$ such that $\hat\psi = \hat h_\star$. 
Then $\Gamma hge_{n+1} = \psi(P)$, and $\hat h\hat\Gamma\hat h^{-1} = \hat\Gamma$, and 
$$\hat\Gamma (\hat h\hat f \hat h^{-1})\hat h\hat g = \hat\Gamma\hat h\hat f\hat g = \hat h\hat\Gamma\hat f\hat g = \hat h\hat\Gamma\hat gs =\hat\Gamma\hat h\hat gs.$$
Therefore, we have that 
\begin{equation}\label{E:8}
s(\hat\psi\hat\phi\hat\psi^{-1},\psi(P);\hat\Gamma\hat h\hat g) = s(\hat\phi,P;\hat\Gamma\hat g).
\end{equation}
Now $d(\psi\phi\psi^{-1})_{\psi(P)} = d\psi_P d\phi_Pd\psi_P^{-1}$, and so Formulas \ref{E:6} and \ref{E:8} imply that
\begin{equation}\label{E:9}
\nu(\hat\psi\hat\phi\hat\psi^{-1},\psi(P)) = \nu(\hat\phi,P).
\end{equation}

For example, suppose $\hat\psi\hat\phi\hat\psi^{-1} = -\hat\phi$.  Then
$$\nu(\hat\phi,P) = \nu(\hat\psi\hat\phi\hat\psi^{-1},\psi(P)) =\nu(-\hat\phi,\psi(P))=-\nu(\hat\phi,\psi(P)).$$
Therefore, the sum of the $\nu$-terms of $\hat\phi$ over the isolated fixed points of $\phi$ cancel in pairs, since each such $\nu$-term is nonzero, 
and so $\phi$ has an even number of isolated fixed points. 

We next work to derive a more useful formula for $s(\hat\phi, P;\hat\Gamma\hat g)$. 
Let $x$ be an element of $H^n$ such that $\Gamma x = P$. 
Then $\Gamma x = P = \phi(P) = \Gamma fx$, and so there is a unique element $\gamma$ of $\Gamma$ 
such that $\gamma f x = x$. 
We choose $g$ in ${\rm SO}^+(n,1)$ so that $ge_{n+1} = x$.  Then $\Gamma ge_{n+1} = \Gamma x = P$ as above. 
We have that $\gamma fge_{n+1} = \gamma f x = x = ge_{n+1}$. Hence $g^{-1}\gamma f g e_{n+1} = e_{n+1}$. 
Let $\hat\gamma$ be the unique element of $\hat\Gamma$ that lifts $\gamma$. 
Then $\hat g^{-1}\hat\gamma\hat f \hat g$ is in ${\rm Spin}(n)$,  
and $\hat\Gamma\hat g(\hat g^{-1}\hat\gamma\hat f \hat g) = \hat\Gamma \hat f\hat g$.  Therefore, we have that
\begin{equation}\label{E:10}
s(\hat\phi, P;\hat\Gamma\hat g) = \hat g^{-1}\hat\gamma\hat f \hat g.
\end{equation}

The linear transformation $d\phi_P$ acts as a rotation on the tangent space $T_P(M)$ 
with nonzero rotation angles $\theta_1,\ldots,\theta_m$ module $2\pi$. 
We have that 
\begin{equation}\label{E:11}
|\det(I-d\phi_P)| = \prod_{j=1}^m(1-e^{{\bf i}\theta_j})(1-e^{-{\bf i}\theta_j}) = \prod_{j=1}^m 4 \sin^2(\theta_j/2).
\end{equation}
To find the rotation angles of $d\phi_P$, 
observe that $\gamma f$ acts as a rotation on the tangent space $T_x(H^4)$ by the same angles, 
and so the eigenvalues of $\gamma f$ are $1, e^{\pm {\bf i}\theta_1},\ldots, e^{\pm {\bf i}\theta_m}$. 
Substituting Formulas \ref{E:10} and \ref{E:11} into Formula \ref{E:6}, gives the formula:
\begin{equation}\label{E:12}
\nu(\hat \phi,P) = \frac{{\rm tr}(\Delta_n^+(\hat g^{-1}\hat\gamma\hat f \hat g)) - {\rm tr}(\Delta_n^-(\hat g^{-1}\hat\gamma\hat f \hat g))}{\prod_{j=1}^m 4\sin^2(\theta_j/2)}.
\end{equation}

By Formulas \ref{E:11} and \ref{E:12} and Formula 8.41 of \cite{A-B}, we have

\begin{equation}\label{E:13}
{\rm tr}(\Delta_n^+(\hat g^{-1}\hat\gamma\hat f \hat g)) -{\rm tr}(\Delta_n^-(\hat g^{-1}\hat\gamma\hat f \hat g)) =\pm {\bf i}^m \prod_{j=1}^m 2\sin(\theta_j/2).
\end{equation}

Squaring both sides of Formula \ref{E:13}, and substituting into Formula \ref{E:12} yields

\begin{equation}\label{E:14}
\nu(\hat \phi,P) = \frac{(-1)^m}{{\rm tr}(\Delta_n^+(\hat g^{-1}\hat\gamma\hat f \hat g)) - {\rm tr}(\Delta_n^-(\hat g^{-1}\hat\gamma\hat f \hat g))}.
\end{equation}

We represent ${\rm Spin}^+(4,1)$ by the matrix group ${\rm SU}(1,1;\mathbb{H})$ defined in $\S \ref{S:2}$.

\begin{theorem}\label{T:7.3} 
Let $\phi$ be an orientation-preserving isometry of a spin, closed, hyperbolic $4$-manifold $M = \Gamma\backslash H^4$, 
and let $\hat \phi$ be a lift of $\phi$ to a symmetry of a spin structure $\hat\Gamma\backslash {\rm SU}(1,1;\mathbb{H})$ on $M$. 
Let $f$ be an element of ${\rm SO}^+(4,1)$ such that $\phi = f_\star$, and let $\hat f$ in ${\rm SU}(1,1;\mathbb{H})$ 
be the lift of $f$ such that $\hat\phi = \hat f_\star$. 
Let $P$ be an isolated fixed point of $\phi$, and let $x$ be an element of $H^4$ such that $\Gamma x = P$. 
Let $g$ be an element of ${\rm SO}^+(4,1)$ such that $ge_5 = x$, 
and let $\hat g$ in ${\rm SU}(1,1;\mathbb{H})$ be a lift of $g$. 
Let $\gamma$ be the unique element of $\Gamma$ such that $\gamma fx = x$, 
and let $\hat\gamma$ in ${\rm SU}(1,1;\mathbb{H})$ be the unique element of $\hat\Gamma$ that lifts $\gamma$. 
Then $\hat g^{-1}\hat\gamma\hat f \hat g = {\rm diag}(p,q)$, with $p$ and $q$ unit quaternions, and  
$$\nu(\hat\phi,P) = \frac{1}{2({\rm Re}(p) -{\rm Re}(q))}.$$
\end{theorem}
\begin{proof}
In \S 6 of \cite{RRT}, we define an isomorphism $\psi:{\mathbb C}\ell(4) \to {\mathbb C}(4)$ of complex algebras. 
The complex spin representation $\Delta_4:{\rm Spin}(4) \to {\mathbb C}(4)$ is the restriction of $\psi$. 
Let $\omega = e_1\cdots e_4$ in the Clifford algebra ${\mathbb C}\ell(4)$. 
In \S 4 of \cite{RRT}, we define a matrix $C$ in ${\mathbb C}(4)$ by $C = {\bf i}^2\Delta_4(\omega)$. Then $C^2 = I$. 
Let $W^+$ and $W^-$ be the $+1$ and $-1$ eigenspaces of $C$. 
Then $\Delta_4^+$ and $\Delta_4^-$ are the complex representations of ${\rm Spin}(4)$ that are obtained by restricting 
the action of ${\rm Spin}(4)$ on $\mathbb C^4$ via $\Delta_4$ to $W^+$ and $W^-$ respectively. 

In \S 4 of \cite{RRT}, we define a retraction $\rho^+: {\mathbb C}\ell(4,1)\to {\mathbb C}\ell(4)$ of complex Clifford algebras 
by $\rho^+(e_j) = e_j$ for $j =1,\ldots, 4$ and $\rho^+(e_5) = {\bf i}^2\omega$. 
The complex spin representation $\Delta_{4,1}:{\rm Spin}^+(4,1) \to {\mathbb C}(4)$ is the restriction of $\psi \rho^+$,  
and extends $\Delta_4$. 
That $\hat g^{-1}\hat\gamma\hat f \hat g = {\rm diag}(p,q)$, with $p$ and $q$ unit quaternions, follows from Theorem 6.3 of \cite{RRT}.

From the discussion in \S 6 of \cite{RRT}, we have that
$$C = {\bf i}^2\psi(\omega) = \psi({\bf i}^2\omega)=\psi(\rho^+(e_5))=\Psi_2(J) = {\rm diag}(1,1,-1,-1).$$
Therefore $W^+ ={\rm Span}\{e_1,e_2\}$ and $W^- ={\rm Span}\{e_3,e_4\}$.

Every quaternion can be written in the form $a+b{\bf j}$ for unique $a, b$ in $\mathbb C$. 
Define monomorphisms $\Psi_1: \mathbb H \to \mathbb C(2)$ and $\Psi_2: \mathbb H(2) \to \mathbb C(4)$ 
of real algebras by
$$\Psi_1(a+b{\bf j}) =\left(\begin{array}{cc}a&b\\-\overline{b}&\overline{a}\end{array}\right)\ \ \text{and}\ \ 
\Psi_2\left(\begin{array}{cc}\alpha&\beta\\\gamma&\delta\end{array}\right) = \left(\begin{array}{cc}\Psi_1(\alpha)&\Psi_1(\beta)\\ \Psi_1(\gamma)&\Psi_1(\delta)\end{array}\right).$$
By Lemma 6.1 of \cite{RRT}, we may replace $\Delta_{4,1}$ by $\Psi_2|_{{\rm SU}(1,1;\mathbb{H})}$. Then 
$${\rm tr}(\Delta_4^+({\rm diag}(p,q))) = {\rm tr}(\Psi_1(p))= 2{\rm Re}(p),$$
$${\rm tr}(\Delta_4^-({\rm diag}(p,q))) = {\rm tr}(\Psi_1(q)) = 2{\rm Re}(q).$$
The result now follows from Formula \ref{E:14}.
\end{proof}

We next work toward finding a more computationally efficient formula for $\nu(\hat\phi, P)$. 

\begin{lemma}\label{L:7.4} 
If $A \in {\rm SU}(1,1;\mathbb H)$ with row vectors $(a,b)$ and $(c,d)$, then $|b| = |c|$. 
\end{lemma}
\begin{proof}
As $A^*JA = J$, we have $A^{-1}= JA^*J$ with row vectors $(\overline a,-\overline c)$ and $(-\overline b, \overline d)$. 
We have that
$$|a|^2-|b|^2  = (AA^{-1})_{11} = (A^{-1}A)_{11} = |a|^2-|c|^2.$$

\vspace{-.2in}
\end{proof}

Our new general formula for computing $\nu(\hat\phi, P)$ when $\dim M = 4$ follows:  

\begin{theorem}\label{T:7.5} 
Let $\phi$ be an orientation-preserving isometry of a spin,  closed, hyperbolic $4$-manifold $M = \Gamma\backslash H^4$, 
and let $\hat \phi$ be a lift of $\phi$ to a symmetry of a spin structure $\hat\Gamma\backslash {\rm SU}(1,1;\mathbb{H})$ on $M$. 
Let $P$ be an isolated fixed point of $\phi$, and let $x$ be an element of $H^4$ such that $\Gamma x = P$. 
Let $\Phi$ be the unique element of ${\rm SO}^+(4,1)$ such that $\phi = \Phi_\star$ and $\Phi x = x$, 
and let $\hat \Phi$ in ${\rm SU}(1,1;\mathbb{H})$ be the lift of $\Phi$ such that $\hat\phi = \hat \Phi_\star$. 
Then
$$\nu(\hat\phi,P) = \frac{x_5}{2({\rm Re}(\hat \Phi_{11}) -{\rm Re}(\hat\Phi_{22}))}.$$
\end{theorem}
\begin{proof}
Let $g$ be in ${\rm SO}^+(4,1)$ such that $ge_5 = x$, and let $\hat g$ in ${\rm SU}(1,1;\mathbb H)$ be a lift of $g$. 
Then $\hat g^{-1}\hat\Phi\hat g = {\rm diag}(p,q)$ with $p$ and $q$ unit quaternions by Theorem \ref{T:7.3}.
Hence $\hat\Phi = \hat g\,{\rm diag}(p,q)\hat g^{-1}$. Let $\hat g$ have row vectors $(a,b)$ and $(c,d)$. 
Then 
$$\hat \Phi = \left(\begin{array}{cc} a & b \\ c & d\end{array}\right)\left(\begin{array}{cc} p & 0 \\ 0 & q\end{array}\right) 
\left(\begin{array}{cc} \overline a & -\overline c \\ -\overline b & \overline d\end{array}\right) 
= \left(\begin{array}{cc} ap\overline a -bq\overline b& \ast \\ \ast & -cp\overline c+dq\overline d\end{array}\right).$$
Hence
\begin{eqnarray*}{\rm Re}(\hat\Phi_{11}) - {\rm Re}(\hat\Phi_{22}) & = &{\rm Re}(ap\overline a - bq\overline b)-{\rm Re}(-cp\overline c+dq\overline d) \\ 
& = & |a|^2{\rm Re}(p) - |b|^2{\rm Re}(q) + |c|^2{\rm Re}(p) - |d|^2{\rm Re}(q) \\ 
& = & (|a|^2+|c|^2){\rm Re}(p) - (|b|^2+|d|^2){\rm Re}(q).\end{eqnarray*}

By Lemma \ref{L:2.1}, we have that 
$$|a|^2+|c|^2 = |a|^2+|b|^2 =  |b|^2+|d|^2.$$
From the formula for the double-covering epimorphism $\eta: {\rm SU}(1,1;\mathbb H) \to {\rm SO}^+(4,1)$ in Table \ref{ta:1}, 
we have that $|a|^2+|b|^2 = g_{55} =x_5$. 
Therefore 
$${\rm Re}(\hat\Phi_{11}) - {\rm Re}(\hat\Phi_{22}) = x_5({\rm Re}(p) - {\rm Re}(q)).$$
The result now follows from Theorem \ref{T:7.3}.
\end{proof}

\section{Spin numbers of the Davis Hyperbolic $4$-Manifold}\label{S:8}  

In this section, we describe our computation of all the spin numbers of the fully symmetric spin structure $\hat\Gamma\backslash{\rm Spin}^+(4,1)$ of the Davis hyperbolic 4-manifold $M = \Gamma\backslash H^4$. 
The group $G$ of orientation-preserving isometries of $M$ has 34 conjugacy classes, 
with 29 conjugacy classes consisting of isometries with only isolated fixed points, 
and the remaining 5 conjugacy classes consisting of isometries with only non-isolated fixed points. 
We found all the fixed points of elements of $G$ by considering the simplicial action of $G$ on the triangulation of $M$ 
induced by the second barycentric subdivision of the regular 120-cell $\mathcal C$.

The double cover $\hat G$ of $G$ contains 54 conjugacy classes. Forty of these conjugacy classes pair off into 20 pairs consisting of a conjugacy class and its minus conjugacy class.  These 40 conjugacy classes of $\hat G$ project to 20 conjugacy classes of $G$. 
The remaining 14 conjugacy classes of $\hat G$ are equal to their minus conjugacy class. 
These 14 conjugacy classes of $\hat G$ project to 14 conjugacy classes of $G$. 

For simplicity, we regard $2I \times 2I$ as a subgroup of $\hat G$ of index 2 with $\hat\sigma_\star$ 
a representative for the other coset. 
Each conjugacy class of $2I\times 2I$ is the Cartesian product of two conjugacy classes of $2I$. 
We denoted the conjugacy classes of $2I$ by ${\bf 1, 2, 3, 4, 5A, 5B, 6, 10A, 10B}$ 
with the numerical part equal to the order of each element of the class. 
The conjugacy classes of $2I$ are determined by the real part of each element of a class listed in Table \ref{ta:2}. 
We denote the real part of each element of a conjugacy class ${\bf x}$ of $2I$ by ${\rm Re}({\bf x})$.
 
 Each conjugacy class of $\hat G$ that is not equal to its minus conjugacy class is either a conjugacy class of $2I\times 2I$ 
 or the union of two conjugacy classes of $2I\times 2I$ that are transposed by conjugating by $\hat\sigma_\star$. 
 Five of the remaining 14 conjugacy classes of $\hat G$ are of this same form. The remaining 9 conjugacy classes of $\hat G$ 
 are subsets of the coset $(2I\times 2I)\hat\sigma_\star$. Each of these conjugacy classes is the set $[{\bf 1}\times{\bf x}]$ 
 of all the conjugates of the elements of the set $({\bf 1}\times{\bf x})\hat\sigma_\star$, where ${\bf x}$ is a conjugacy class of $2I$. 
 In particular, $[{\bf 1}\times{\bf 1}]$ is the set of all the conjugates of $\hat\sigma_\star$. 
 That all the elements of the set $({\bf 1}\times{\bf x})\hat\sigma_\star$ are conjugate follows from the fact 
 that $(\alpha^{-1}(q),q)$ commutes with $\hat\sigma_\star$ for each $q$ in $2I$ by Theorem \ref{T:4.1}. 
 
 All the conjugacy classes of $\hat G$ are listed in Table \ref{ta:6} with a horizontal line separating the two types of classes. 
 The second column lists the orders of the elements in each class in the first column. 
 The third column lists the number of elements in each class. 
 The fourth column lists the number of fixed points of the action on $M$. 
 The fifth column lists the spin number of the class in the first column. 
 The last column lists the corresponding minus class. 
 The spin number of a minus class is minus the spin number of the class.  
 The order of a class in the last column above the horizontal line is the least common multiple 
 of the numerical parts of its name.

\begin{table} 
$$\begin{array}{|c|rrccc|} \hline
\text{Conjugacy Class} & \text{Ord}\!\!\! &\text{Size} & \!\!\text{\# FP}\!\! & \text{Spin \#}\!\!\! & -\,\text{Conjugacy Class} \\ \hline
{\bf 1}\times{\bf 1} & 1 & 1 &  \infty &  0 & {\bf 2}\times{\bf 2}   \\ 
{\bf 1}\times{\bf 3}+{\bf 3}\times{\bf 1}& 3 & 40 &  2 & 0 & {\bf 2}\times{\bf 6}+{\bf 6}\times{\bf 2}  \\ 
{\bf 3}\times{\bf 3}& 3 & 400&  \infty  & 0 & {\bf 6}\times{\bf 6}  \\ 
{\bf 1}\times{\bf 4}+{\bf 4}\times{\bf 1}& 4 & 60&  2 & 0 & {\bf 2}\times{\bf 4}+{\bf 4}\times{\bf 2}  \\ 
{\bf 1}\times{\bf 5A}+{\bf 5B}\times{\bf 1}& 5 & 24&  26 & 5\sqrt{5} & {\bf 2}\times{\bf 10B}+{\bf 10A}\times{\bf 2}  \\ 
{\bf 1}\times{\bf 5B}+{\bf 5A}\times{\bf 1}& 5 & 24&  26 & -5\sqrt{5} & {\bf 2}\times{\bf 10A}+{\bf 10B}\times{\bf 2}  \\ 
{\bf 5A}\times{\bf 5B} & 5 & 144&  6 & 0 & {\bf 10B}\times{\bf 10A}  \\ 
{\bf 5B}\times{\bf 5A} & 5 & 144&  6 & 0 & {\bf 10A}\times{\bf 10B}  \\ 
{\bf 5A}\times{\bf 5A}+{\bf 5B}\times{\bf 5B}& 5 & 288 &  \infty & 0 & \!\!\!{\bf 10A}\times{\bf 10A}+{\bf 10B}\times{\bf 10B} \\ 
{\bf 1}\times{\bf 6}+{\bf 6}\times{\bf 1}& 6 & 40 &  2 & 0 & {\bf 2}\times{\bf 3}+{\bf 3}\times{\bf 2} \\ 
{\bf 1}\times{\bf 10A}+{\bf 10B}\times{\bf 1}& 10 & 24 &  2 & \sqrt{5} & {\bf 2}\times{\bf 5B}+{\bf 5A}\times{\bf 2} \\ 
{\bf 1}\times{\bf 10B}+{\bf 10A}\times{\bf 1}& 10 & 24 &  2 & -\sqrt{5} & {\bf 2}\times{\bf 5A}+{\bf 5B}\times{\bf 2} \\
{\bf 5A}\times{\bf 10B}+{\bf 10A}\times{\bf 5B}& 10 & 288 &  12 & -2\sqrt{5} & {\bf 5B}\times{\bf 10A}+{\bf 10B}\times{\bf 5A} \\ 
{\bf 3}\times{\bf 4}+{\bf 4}\times{\bf 3}& 12 & 1200 &  2 & 0 & {\bf 4}\times{\bf 6}+{\bf 6}\times{\bf 4} \\ 
{\bf 3}\times{\bf 5A}+{\bf 5B}\times{\bf 3}& 15 & 480 &  2 & -\sqrt{5} & {\bf 6}\times{\bf 10B}+{\bf 10A}\times{\bf 6} \\ 
{\bf 3}\times{\bf 5B}+{\bf 5A}\times{\bf 3}& 15 & 480 &  2 & \sqrt{5} & {\bf 6}\times{\bf 10A}+{\bf 10B}\times{\bf 6} \\ 
{\bf 4}\times{\bf 5A}+{\bf 5B}\times{\bf 4}& 20 & 720 &  2 & -\sqrt{5} & {\bf 4}\times{\bf 10B}+{\bf 10A}\times{\bf 4} \\ 
{\bf 4}\times{\bf 5B}+{\bf 5A}\times{\bf 4}& 20 & 720 &  2 & \sqrt{5} & {\bf 4}\times{\bf 10A}+{\bf 10B}\times{\bf 4} \\ 
{\bf 3}\times{\bf 10A}+{\bf 10B}\times{\bf 3}& 30 & 480 &  2 & \sqrt{5} & {\bf 5A}\times{\bf 6}+{\bf 6}\times{\bf 5B} \\ 
{\bf 3}\times{\bf 10B}+{\bf 10A}\times{\bf 3}& 30 & 480 &  2 & -\sqrt{5} & {\bf 5B}\times{\bf 6}+{\bf 6}\times{\bf 5A} \\ \hline
{\bf 1}\times{\bf 2}+{\bf 2}\times{\bf 1}& 2 & 2 &  122 & 0 & {\bf 1}\times{\bf 2}+{\bf 2}\times{\bf 1} \\ \relax
[{\bf 1}\times{\bf 2}] & 2 & 120 &  10 & 0 & [{\bf 1}\times{\bf 2}]\\ \relax
[{\bf 1}\times{\bf 1}] & 4 & 120 &  \infty & 0 & [{\bf 1}\times{\bf 1}]\\ 
{\bf 4}\times{\bf 4} & 4 & 900 &  \infty & 0 & {\bf 4}\times{\bf 4}\\ 
{\bf 3}\times{\bf 6}+{\bf 6}\times{\bf 3} & 6 & 800 &  8 & 0 & {\bf 3}\times{\bf 6}+{\bf 6}\times{\bf 3}\\ \relax
[{\bf 1}\times{\bf 6}] & 6 & 2400 &  4 & 0 & [{\bf 1}\times{\bf 6}]\\ \relax
[{\bf 1}\times{\bf 4}] & 8 & 3600 &  2 & 0 & [{\bf 1}\times{\bf 4}]\\ 
{\bf 5A}\times{\bf 10A}+{\bf 10B}\times{\bf 5B} & 10 & 288 &  2 & 0 &{\bf 5A}\times{\bf 10A}+{\bf 10B}\times{\bf 5B}\\
{\bf 5B}\times{\bf 10B}+{\bf 10A}\times{\bf 5A} & 10 & 288 &  2 & 0 &{\bf 5B}\times{\bf 10B}+{\bf 10A}\times{\bf 5A} \\ \relax
[{\bf 1}\times{\bf 10A}] & 10 & 1440 &  0 & 0 &[{\bf 1}\times{\bf 10A}] \\ \relax
[{\bf 1}\times{\bf 10B}] & 10 & 1440 &  0 & 0 &[{\bf 1}\times{\bf 10B}] \\ \relax
[{\bf 1}\times{\bf 3}] & 12 & 2400 &  0 & 0 &[{\bf 1}\times{\bf 3}] \\ \relax
[{\bf 1}\times{\bf 5A}] & 20 & 1440 &  4 & 0 &[{\bf 1}\times{\bf 5A}] \\ \relax
[{\bf 1}\times{\bf 5B}] & 20 & 1440 &  4 & 0 &[{\bf 1}\times{\bf 5B}] \\ \hline
\end{array}$$
\caption{Table of the conjugacy classes of $\hat G$ and their spin numbers}\label{ta:6}
\end{table}

\begin{lemma}\label{L:8.1} 
The minus conjugacy class of a conjugacy class of $\hat G$ is the class in the same row and in the last column of {\rm Table $\ref{ta:6}$}. 
In particular, the conjugacy classes of $\hat G$ listed in {\rm Table $\ref{ta:6}$} below the horizontal line are equal to their minus conjugacy class. 
Therefore, the spin numbers of these classes are $0$ by {\rm Lemma $\ref{L:7.1}$}. 
\end{lemma}
\begin{proof}
By Theorem \ref{T:4.1}, we have that $\hat\sigma_\star (1,-1) \hat\sigma_\star^{-1} = (-1,1)$. 
Multiplying this equation on the left by $(1,-1)$ and on the right by $\hat\sigma_\star$ gives
$(1,-1)\hat\sigma_\star(1,-1) = -\hat\sigma_\star$. 
Hence, if $(p,q)$ is an element of $2I\times 2I$, then $(1,-1)(p,q)\hat\sigma_\star(1,-1) = -(p,q)\hat\sigma_\star$. 
Therefore $[{\bf 1}\times {\bf x}] = -[{\bf 1}\times {\bf x}]$ for each conjugacy class {\bf x} of $2I$. 

From Table \ref{ta:2}, we see that $-{\rm Re}({\bf 1}) = {\rm Re}({\bf 2})$, $-{\rm Re}({\bf 3}) = {\rm Re}({\bf 6})$, 
$-{\rm Re}({\bf 4}) = {\rm Re}({\bf 4})$, $-{\rm Re}({\bf 5A}) = {\rm Re}({\bf 10B})$, and $-{\rm Re}({\bf 5B}) = {\rm Re}({\bf 10A})$. 
Therefore $-{\bf 1} = {\bf 2}$, $-{\bf 3} = {\bf 6}$, 
$-{\bf 4} = {\bf 4}$, $-{\bf 5A} = {\bf 10B}$, and $-{\bf 5B} = {\bf 10A}$.
Hence, multiplying the first column of Table \ref{ta:6} by $-1 = (-1,-1)$ gives the last column of Table \ref{ta:6}. 
\end{proof}

The formula in Lemma \ref{L:8.2}  easily computes the spin numbers of the 24 conjugacy classes of $\hat G$ 
above the horizontal line in Table \ref{ta:6} that fix exactly two points of $M$.

\begin{lemma}\label{L:8.2} 
Let ${\bf C} = {\bf x}\times {\bf y}+ {\bf y'}\times {\bf x'}$ be a conjugacy class of $\hat G$ contained in $2I \times 2I$ 
that fixes exactly two points of $M$. 
Then 
$${\rm Spin}({\bf C},M) = \frac{1}{2({\rm Re}({\bf x})-{\rm Re}({\bf y}))}+\frac{1}{2({\rm Re}({\bf y'})-{\rm Re}({\bf x'}))}.$$
\end{lemma}
\begin{proof}
The fixed points of ${\bf C}$ are the canonical points $C =\Gamma e_5$ and $A$ that are fixed by all the elements of $2I\times 2I$.
By Formulas \ref{E:3} and \ref{E:4}, we have that 
$${\rm Spin}({\bf C},M) = \nu({\bf x}\times{\bf y},C)+ \nu({\bf x}\times{\bf y},A).$$
By Formula \ref{E:9}, we have that 
$$\nu({\bf x}\times {\bf y},A) =\nu(\hat\sigma_\star({\bf x}\times {\bf y})\hat\sigma_\star^{-1},\sigma_\star(A)) = \nu({\bf y'}\times {\bf x'},C).$$
The result now follows Theorem \ref{T:7.5} with $x = e_5$.
\end{proof}

Let ${\bf C}$ be a conjugacy class of $\hat G$ above the horizontal line in Table \ref{ta:6} with only isolated fixed points and more than 2 fixed points. 
Then the number of fixed points of ${\bf C}$ is either 26, 6, or 12. 

1) Suppose that ${\bf C}$ fixes exactly 26 points.  Then the fixed points are the canonical points $C$ and $A$ and 24 cycles of ridge centers 
of ${\mathcal C}$. 
The value of $\nu({\bf C},C)+ \nu({\bf C},A)$ can be computed using the right-hand side of the formula in Lemma \ref{L:8.2}.

Let $\hat\phi$ be an element of ${\bf C}$.  
Represent $\phi$ by an element $f$ of ${\rm Sym}_0(\mathcal C)$ and a cycle $P$ of ridge centers fixed by 
$\phi = f_\star$ by a ridge center $x$. 
Then $fx$ is an other ridge center in the cycle $P$.  Ridge center cycles consist of five points, and so $P = \{c_1,c_2,c_3,c_4,c_5\}$ 
and there is a sequence of side-pairing transformations $g_{i_1}, \ldots, g_{i_5}$ in $\Gamma$ such that $g_{i_j}(c_j) = c_{j+1}$ with $j$ taken modulo 5,  
cf. Table 2 of \cite{ratcliffe-tschantz:davis}. 
This means that there is a product $\gamma$ of at most four of $g_{i_1}, \ldots, g_{i_5}$ so that $\gamma fx = x$. 
Let $\Phi = \gamma f$. 
Then $\Phi_\star = \phi$. 
Let $\hat\gamma$ be the unique element of $\hat\Gamma$ that lifts $\gamma$, and let $\hat f$ be the unique element of ${\rm SU}(1,1;\mathbb H)$ that lifts $f$ and $\hat\gamma\hat f = \hat\phi$. Then $\hat\Phi = \hat\gamma\hat f$ is the lift of $\Phi$ 
that we use to compute $\nu(\hat\phi, P)$ via Theorem \ref{T:7.5}, cf. \S 10.6 of \cite{RRT}.

2) Suppose that ${\bf C}$ fixes exactly 6 points. 
Then the fixed points are the canonical points $C$ and $A$ and 4 cycles of ridge centers of ${\mathcal C}$. 
The value of $\nu({\bf C},C)+ \nu({\bf C},A)$ can be computed using the right-hand side of the formula in Lemma \ref{L:8.2} with ${\bf y'}\times{\bf x'} = {\bf x}\times{\bf y}$. The $\nu$-values of the ridge cycles are computed as in Case 1. 

3) Suppose that ${\bf C}$ fixes exactly 12 points. 
Then the fixed points are the canonical points $C$ and $A$, and 5 cycles of edge centers of ${\mathcal C}$, and 5 cycles of side centers of ${\mathcal C}$. 
The value of $\nu({\bf C},C)+ \nu({\bf C},A)$ can be computed using the right-hand side of the formula in Lemma \ref{L:8.2}. 
To compute the $\nu$-values of the remaining fixed points, we proceed as in Case 1. 
For a cycle of side centers, the value of $\gamma$ is the side-pairing transformation that maps the side-center $fx$ back to the side-center $x$. 

Each cycle of edge centers consists of 20 points that form the vertices of a regular dodecahedron $P$ in $H^4$.  
Each edge of $P$ is translated to the opposite edge of $P$ by a side-pairing transformation in $\Gamma$. 
The graph whose vertices are the vertices of $P$ and whose edges join vertices that are images of each other 
by a side-pairing transformation is the 1-skeleton of a great stellated dodecaheron $P'$. 
By composing the side-pairing transformations corresponding to an edge path in $P'$ from the vertex $fx$ of $P'$ to the vertex $x$ of $P'$
gives the element $\gamma$ of $\Gamma$ such that $\gamma fx = x$. 

There are only two nontrivial conjugacy classes of $G$ with non-isolated fixed points corresponding to a conjugacy class of $\hat G$ 
above the horizontal line in Table \ref{ta:6} namely, the classes corresponding to ${\bf 3}\times{\bf 3}$ and ${\bf 5A}\times{\bf 5A}+{\bf 5B}\times{\bf 5B}$. 

1) The conjugacy class of $G$ corresponding to ${\bf 3}\times{\bf 3}$ is represented by the isometry $\phi$ of $M$ of order 3 
represented by the permutation $5\times 5$ matrix $\Phi$ corresponding to the 3-cycle $(2,3,4)$.  
The 3-dimensional vector subspace $V$ of $\mathbb R^{4,1}$, 
consisting of the vectors fixed by $\Phi$, is spanned by $e_1, e_5$ and $e_2+e_3+e_4$. 
The subspace $V$ slices through the center of the 120-cell $\mathcal C$ in a dodecagon $P$ having sides alternately between opposite vertices in sides of $\mathcal C$, and edges of $\mathcal C$.  The polygon $P$ is drawn in Figure \ref{F:1}. 
The fixed set $M^\phi$ is a connected, totally geodesic, embedded,  orientable, closed surface of genus 3 
obtained by identifying opposite pairs of sides of $P$ as in Figure \ref{F:1}. 
Hence ${\rm Spin}(\hat\phi, M)= 0$ by Lemma \ref{L:7.2}.

\begin{figure}[t]
\centering
{\includegraphics[width=2.5in]{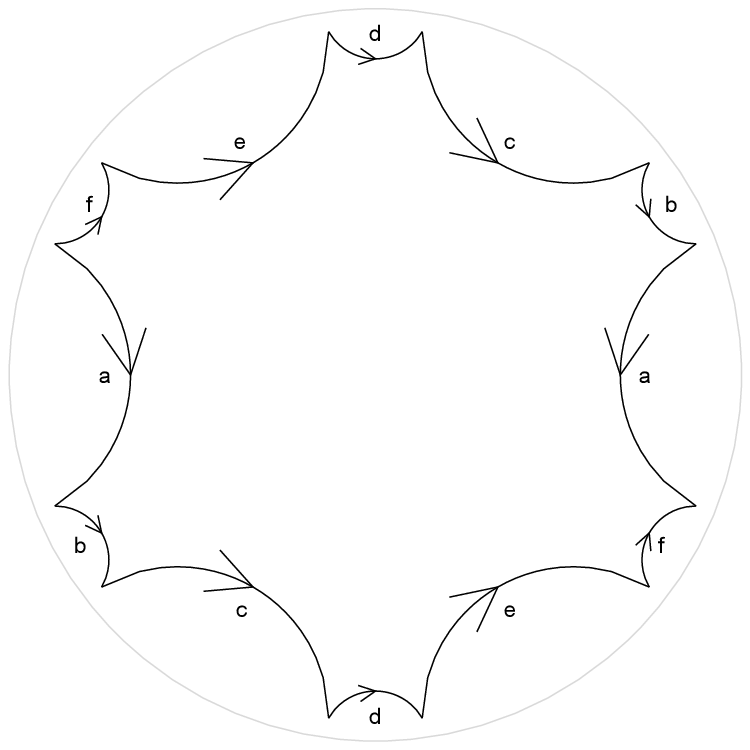}}
\caption{The gluing pattern for $M^\phi$ with $\hat\phi$ in ${\bf 3}\times{\bf 3}$}\label{F:1}
\end{figure}

2) The conjugacy class of $G$ corresponding to ${\bf 5A}\times{\bf 5A}+{\bf 5B}\times{\bf 5B}$ is represented by the isometry $\phi$ of $M$ of order 5 
represented by matrix 
$$\Phi = \frac{1}{2}\left(\begin{array}{ccccc}
2 & 0 & 0 & 0 & 0 \\
0 & 1 & \tau & 1-\tau & 0 \\
0 & \tau & 1-\tau & 1 & 0 \\
0 & \tau - 1 & -1 & -\tau & 0 \\
0 & 0 & 0 & 0 & 2 
\end{array}\right).$$
The 3-dimensional vector subspace $V$ of $\mathbb R^{4,1}$, consisting of the vectors fixed by $\Phi$, is spanned by $e_1, e_5$ and $\tau e_2+e_3$. 
The subspace $V$ 
slices through the center of the 120-cell $\mathcal C$ in a regular decagon $P$, with angles $2\pi/5$, 
having sides between the centers of opposite ridges in sides of the 120-cell $\mathcal C$. 

The fixed set $M^\phi$ has two connected components. 
The first component $M^\phi_1$ is obtained by identifying opposite pairs of sides of $P$ as in Figure \ref{F:2} to give 
a totally geodesic, embedded, orientable, closed surface of genus 2. 

There are also two regular pentagons with angles $\pi/5$ in $M$ that are point-wise fixed by $\phi$. 
Each of these pentagons corresponds to a cycle of ridges of $\mathcal C$ that is cyclically permuted by $\Phi$. 
These two pentagons glue up, according to the pattern in Figure \ref{F:2}, to give the second component $M_2^\phi$. 

The surfaces $M^\phi_1$ and $M^\phi_2$ are isometric, since their gluing patterns are dual to each other. 
The 10 vertices of the two pentagons glue up to the canonical point $A$, so this gives the same 
surface as gluing a regular decagon, with angles $2\pi/5$, centered at $A$ instead of $C$. 
We have that ${\rm Spin}(\hat\phi, M)= 0$ by Lemma \ref{L:7.2}.

\begin{figure}[t]
\centering
{\includegraphics[width=5in]{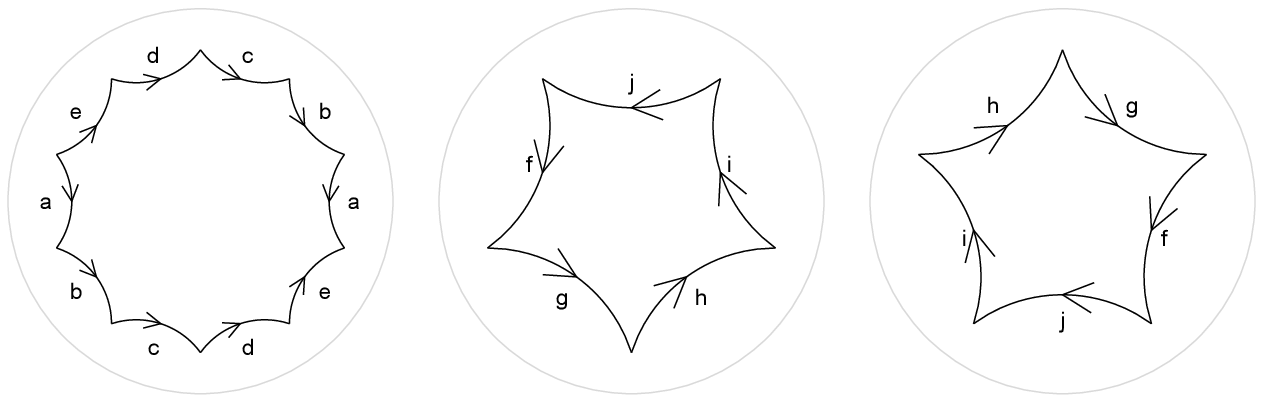}}
\caption{The gluing pattern for $M^\phi$ with $\hat\phi$ in ${\bf 5A}\times{\bf 5A}+{\bf 5B}\times{\bf 5B}$}\label{F:2}
\end{figure}

We end this section by describing the fixed sets of the involutions of $M$. 

1) The conjugacy class of $G$ corresponding to ${\bf 1}\times{\bf 2}+{\bf 2}\times{\bf 1}$ is represented by the antipodal map $\alpha$ 
of $M$ of order 2 represented by ${\diag}(-1,-1,-1,-1,1)$. 
The fixed points of $\alpha$ are $A$ and $C$, the 60 edge-center cycles and the 60 side-center cycles. 

2) The conjugacy class of $G$ corresponding to $[{\bf 1}\times{\bf 2}]$ is represented by the isometry $\alpha\sigma$ of $M$ of order 2. 
The isometry $\alpha\sigma$ has 10 fixed points, each of which is the midpoint 
of a line segment joining the canonical points $A$ and $C$

3) The conjugacy class of $G$ corresponding to $[{\bf 1}\times{\bf 1}]$ is represented by the inside-out isometry $\sigma$ of $M$ of order 2.
The fixed set of $\sigma$ is a connected, totally geodesic, embedded, orientable, closed surface of genus 4  
that is tessellated by 30 regular quadrilaterals with angles $2\pi/5$. 

4) The conjugacy class of $G$ corresponding to ${\bf 4}\times{\bf 4}$ is represented by the isometry $\phi$ of $M$ of order 2 
represented by ${\rm diag}(1,1,-1,-1,1)$. 
The fixed set of $\phi$ is a connected, totally geodesic, embedded,  orientable, closed surface of genus 4. 

\section{The $\hat G$-Index of the Davis Hyperbolic 4-Manifold}\label{S:9}  

In this section, we determine the $\hat G$-index of the fully symmetric spin structure of the Davis Hyperbolic 4-manifold. 

\begin{theorem}\label{T:9.1} 
Let $\hat G$ be the group of symmetries of the fully symmetric spin structure of the Davis hyperbolic 4-manifold $M$, 
and let $\rho^{\pm}: \hat G \to {\rm GL}( \mathcal H^{\pm})$ be the representations of $\hat G$ induced by the action of $\hat G$ 
on the complex vector spaces $\mathcal H^{\pm}$ of positive (negative) spinors of $M$. 
Then 
\begin{enumerate}
\item ${\rm Spin}(\hat G, M)=({\bf 2'}\otimes{\bf 3'})\oplus({\bf 3}\otimes{\bf 2})- ({\bf 2}\otimes{\bf 3})\oplus({\bf 3'}\otimes{\bf 2'})$ in $R(\hat G)$.
\item There is a spinorial representation $\rho$ of $\hat G$ such that 
$$\rho^+ \cong ({\bf 2'}\otimes{\bf 3'})\oplus({\bf 3}\otimes{\bf 2})\oplus \rho\ \ \text{and}\ \ \rho^-\cong ({\bf 2}\otimes{\bf 3})\oplus({\bf 3'}\otimes{\bf 2'})\oplus \rho.$$
\item The dimension of $\mathcal H = \mathcal H^+\oplus\mathcal H^-$ is at least $24$, and is divisible by $8$. 
\end{enumerate}
\end{theorem}
\begin{proof}
1) ${\rm Spin}(\hat G, M) = \rho^+ - \rho^-$ and 
$({\bf 2'}\otimes{\bf 3'})\oplus({\bf 3}\otimes{\bf 2})- ({\bf 2}\otimes{\bf 3})\oplus({\bf 3'}\otimes{\bf 2'})$  
have the same character by our computation of all the spin numbers listed in Table \ref{ta:6}, and so they are equal in $R(\hat G)$. 

2) Factor $\rho^\pm$ into irreducible representations, and let $\rho$ be the direct sum of the common representations. 
Then (1) implies that 
$$\rho^+ \cong ({\bf 2'}\otimes{\bf 3'})\oplus({\bf 3}\otimes{\bf 2})\oplus \rho\ \ \text{and}\ \ \rho^- \cong
({\bf 2}\otimes{\bf 3})\oplus({\bf 3'}\otimes{\bf 2'})\oplus \rho$$
since $({\bf 2'}\otimes{\bf 3'})\oplus({\bf 3}\otimes{\bf 2})$ and $({\bf 2}\otimes{\bf 3})\oplus({\bf 3'}\otimes{\bf 2'})$ are irreducible representations of $\hat G$, and $R(\hat G)$ is a free abelian group on all the isomorphism classes of irreducible representations of $\hat G$. 
Moreover, $\rho$ is spinorial, since  $\rho^+$ is spinorial. 

3) Part (3) follows from (2), since both $({\bf 2'}\otimes{\bf 3'})\oplus({\bf 3}\otimes{\bf 2})$ and $({\bf 2}\otimes{\bf 3})\oplus({\bf 3'}\otimes{\bf 2'})$ are 12-dimensional, and the dimension of every spinorial  irreducible representation of $\hat G$ is divisible by $4$ by our classification in  Table \ref{ta:4}.
\end{proof}

\section{Spin Numbers in the 2-Dimensional Case}\label{S:10}  

Let $\mathbb C(2)$ be the algebra of complex $2\times 2$ matrices. 
If $A$ is in $\mathbb C(2)$, let $A^*$ be the conjugate transpose of $A$.  Let $J = {\rm diag}(1,-1)$, and let 
$${\rm SU}(1,1;\mathbb C) = \{A \in \mathbb C(2): A^*JA = J \ \ \text{and}\ \det A = 1\}.$$
The group ${\rm SU}(1,1;\mathbb C)$ acts on the conformal ball model of hyperbolic 2-space
$$B^2 =\{z \in \mathbb C: |z| < 1\}$$
by linear fractional transformations so that 
$$\left(\begin{array}{cc} a & b \\ c & d\end{array}\right)\cdot z = (az+b)(cz +d)^{-1}.$$

In \S 5 of our paper \cite{RRT}, we defined a double-covering epimorphism 
$$\eta: {\rm SU}(1,1;\mathbb C) \to {\rm SO}^+(2,1).$$
As in \S \ref{S:2}, we now make a small change in the definition of $\eta$ for an aesthetic reason. 
We change the definition of $E_3$ from $J$ to $-J$, and we replace $\rho^-$ by $\rho^+$ throughout \S5 of \cite{RRT}. 
The new $\eta$ is conjugation by $J$ followed by the old $\eta$. 
All the main results of \S5 and \S9 of \cite{RRT} are unaltered by this change. 
The only consequence of this change to \cite{RRT} is that the off-diagonal entries of a lift $\hat M$ in ${\rm SU}(1,1;\mathbb C)$ 
of a matrix $M$ in ${\rm SO}^+(2,1)$, with respect to $\eta$, change sign. 

With this change, $\eta$ is now compatible with stereographic projection $\zeta: B^2 \to H^2$ (cf. Formula 4.5.2 \cite{R}), that is, for all $A \in {\rm SU}(1,1;\mathbb C)$ and all $z \in B^2$, we have $\zeta(A\cdot z) = \eta(A)\zeta(z).$ 

If $z \in \mathbb{C}$, write $z= z_1 +z_2{\bf i}$ with $z_1$ and $z_2$ real numbers.  
The new definition of $\eta$ is given by 
$$\eta\left(\begin{array}{cc} a & b \\ c & d\end{array}\right) = \left(\begin{array}{ccc} 1-2a_2^2+2b_1^2 & -2a_1a_2+2b_1b_2 & 2a_1b_1-2a_2b_2 \\
2a_1a_2+2b_1b_2 & 1-2a_2^2 + 2b_2^2 & 2a_1b_2+2a_2b_1\\
2a_1b_1+2a_2b_2 & 2a_1b_2-2a_2b_1 & 1 + 2b_1^2 + 2b_2^2\end{array}\right).$$
This formula appears to be unbalanced with respect to $a,b,c,d$; however, 
this is not the case, since $c = \overline b$ and $d = \overline a$ by Lemma 5.2 of \cite{RRT}.

\begin{theorem}\label{T:10.1} 
Let $\phi$ be an orientation-preserving isometry of a spin, closed, hyperbolic $2$-manifold $M = \Gamma\backslash H^2$, 
and let $\hat \phi$ be a lift of $\phi$ to a symmetry of a spin structure $\hat\Gamma\backslash {\rm SU}(1,1;\mathbb{C})$ on $M$. 
Let $f$ be an element of ${\rm SO}^+(2,1)$ such that $\phi = f_\star$, and let $\hat f$ in ${\rm SU}(1,1;\mathbb{C})$ 
be the lift of $f$ such that $\hat\phi = \hat f_\star$. 
Let $P$ be an isolated fixed point of $\phi$, and let $x$ be an element of $H^2$ such that $\Gamma x = P$. 
Let $g$ be an element of ${\rm SO}^+(2,1)$ such that $ge_3 = x$, 
and let $\hat g$ in ${\rm SU}(1,1;\mathbb{C})$ be a lift of $g$. 
Let $\gamma$ be the unique element of $\Gamma$ such that $\gamma fx = x$, 
and let $\hat\gamma$ in ${\rm SU}(1,1;\mathbb{C})$ be the unique element of $\hat\Gamma$ that lifts $\gamma$. 
Then $\hat g^{-1}\hat\gamma\hat f \hat g = {\rm diag}(u,\overline u)$, with $u$ a unit complex number, and  
$$\nu(\hat\phi,P) = \frac{1}{2\,{\rm Im}(u){\bf i}}\,.$$
\end{theorem}

\begin{proof}
In \S 5 of \cite{RRT}, we define an isomorphism $\psi:{\mathbb C}\ell(2) \to {\mathbb C}(2)$ of complex algebras. 
The complex spin representation $\Delta_2:{\rm Spin}(2) \to {\mathbb C}(2)$ is the restriction of $\psi$. 
Let $\omega = e_1 e_2$ in the Clifford algebra ${\mathbb C}\ell(2)$. 
Let $W^+$ and $W^-$ be the $+1$ and $-1$ eigenspaces of the matrix $C = {\bf i}\Delta_2(\omega)$. 
Then $\Delta_2^+$ and $\Delta_2^-$ are the complex representations of ${\rm Spin}(2)$ that are obtained by restricting 
the action of ${\rm Spin}(2)$ on $\mathbb C^2$ via $\Delta_2$ to $W^+$ and $W^-$ respectively. 

In \S 4 of \cite{RRT}, we define a retraction $\rho^+: {\mathbb C}\ell(2,1)\to {\mathbb C}\ell(2)$ of complex Clifford algebras 
by $\rho^+(e_j) = e_j$ for $j =1,2$ and $\rho^+(e_3) = {\bf i}\omega$. 
The complex spin representation $\Delta_{2,1}:{\rm Spin}^+(2,1) \to {\mathbb C}(2)$ is the restriction of $\psi \rho^+$,  
and extends $\Delta_2$. 
That $\hat g^{-1}\hat\gamma\hat f \hat g = {\rm diag}(u,\overline u)$, with $u$ a unit complex number, follows from Theorem 5.2 of \cite{RRT}.

Now we have
$$C = {\bf i}\psi(\omega) = \psi({\bf i}\omega) = \psi(\rho^+(e_3) )= E_3= -J = {\rm diag}(-1,1).$$
Therefore $W^+ ={\rm Span}\{e_2\}$ and $W^- ={\rm Span}\{e_1\}$.

By Theorem 5.1 of \cite{RRT}, we may replace $\Delta_{2,1}$ by the inclusion of ${\rm SU}(1,1;\mathbb{C})$ into $\mathbb C(2)$. Then 
$${\rm tr}(\Delta_2^+({\rm diag}(u,\overline u))) = \overline u \ \ \ \text{and}\ \ \ 
{\rm tr}(\Delta_2^-({\rm diag}(u,\overline u)))  =  u.$$
By Formula \ref{E:14}, we have 
$$\nu(\hat\phi,P) = \frac{-1}{\overline u - u} = \frac{1}{2\,{\rm Im}(u){\bf i}}\,.$$

\vspace{-.25in}
\end{proof}

Our new general formula for computing $\nu(\hat\phi, P)$ when $\dim M = 2$ follows:  

\begin{theorem}\label{T:10.2} 
Let $\phi$ be an orientation-preserving isometry of a spin, closed, hyperbolic $2$-manifold $M = \Gamma\backslash H^2$, 
and let $\hat \phi$ be a lift of $\phi$ to a symmetry of a spin structure $\hat\Gamma\backslash {\rm SU}(1,1;\mathbb{C})$ on $M$. 
Let $P$ be an isolated fixed point of $\phi$, and let $x$ be an element of $H^2$ such that $\Gamma x = P$. 
Let $\Phi$ be the unique element of ${\rm SO}^+(2,1)$ such that $\phi = \Phi_\star$ and $\Phi x = x$, 
and let $\hat \Phi$ in ${\rm SU}(1,1;\mathbb{C})$ be the lift of $\Phi$ such that $\hat\phi = \hat \Phi_\star$. 
Then
$$\nu(\hat\phi,P) = \frac{x_3}{2\,{\rm Im}(\hat\Phi_{11}){\bf i}}\,.$$
\end{theorem}
\begin{proof}
Let $g$ be in ${\rm SO}^+(2,1)$ such that $ge_3 = x$, and let $\hat g$ in ${\rm SU}(1,1;\mathbb C)$ be a lift of $g$. 
Then $\hat g^{-1}\hat\Phi\hat g = {\rm diag}(u,\overline u)$ with $u$ a unit complex number by Theorem \ref{T:10.1}.
Hence $\hat\Phi = \hat g\,{\rm diag}(u,\overline u)\hat g^{-1}$. Let $\hat g$ have row vectors $(a,b)$ and $(\overline b,\overline a)$. 
Then 
$$\hat \Phi = \left(\begin{array}{cc} a & b \\ \overline b & \overline a\end{array}\right)\left(\begin{array}{cc} u & 0 \\ 0 & \overline u\end{array}\right) 
\left(\begin{array}{cc} \overline a & -b \\ -\overline b & a\end{array}\right) 
= \left(\begin{array}{cc} |a|^2u -|b|^2\overline u& \ast \\ \ast & -|b|^2u+|a|^2\overline u\end{array}\right).$$
Hence
$$\hat\Phi_{11} - \hat\Phi_{22} = |a|^2u -|b|^2\overline u + |b|^2u-|a|^2\overline u = (|a|^2+|b|^2)(u-\overline u).$$

From the formula for the double-covering epimorphism $\eta: {\rm SU}(1,1;\mathbb C) \to {\rm SO}^+(2,1)$, 
we have that $|a|^2+|b|^2 = 1+2|b|^2 = g_{33} = x_3$. 
Therefore 
$$2\,{\rm Im}(\hat\Phi_{11}) = \hat\Phi_{11} - \hat\Phi_{22} = x_3(u-\overline u) = 2 x_3 {\rm Im}(u).$$
The result now follows from Theorem \ref{T:10.1}.
\end{proof}

\subsection*{Corrections}
For the case of isolated fixed points, with no eigenvalue equal to $-1$, 
the Atiyah-Hirzebruch version of the $G$-spin theorem given in Formula 8 of \cite{A-H}
differs from Shanahan's version of the $G$-spin theorem on p 174 of \cite{S} by a factor of $(-1)^\ell$, with $2\ell = n$,  
because Formula 7 in \cite{A-H}, for $z = e^{{\bf i}\theta}$, gives $(2\sinh((x_j- {\bf i}\theta)/2))^{-1}$
whereas Shanahan uses $(2\sinh((x_j+ {\bf i}\theta)/2))^{-1}$ instead, cf. p 174 of \cite{S}, 
and so the relevant degree 0 terms $\mp {\bf i}\csc(\theta/2)/2$ of the Taylor series of $(2\sinh((x_j\pm {\bf i}\theta)/2))^{-1}$ differ in sign. 

On a more basic level, the difference factor $(-1)^\ell$ is explained by comparing the definitions of $\Delta_n^\pm$ on p 483 of \cite{A-B} 
and p 48 of \cite{S}. Let $\omega = e_1\cdots e_n$. 
Using the fact that the operator $Q_1\cdots Q_\ell$ is equal to multiplication by ${\bf i}^\ell\omega$, 
we see that ${\bf i}^\ell\omega$ acts on Shanahan's representation space $\Delta_n$ by multiplication by $(-1)^\ell$.   
Therefore Shanahan's $\Delta_n^+$ is Atiyah-Bott's $\Delta_n^-$ if $\ell$ is odd or Atiyah-Bott's $\Delta_n^+$ if $\ell$ is even, 
and so we have a difference factor of $(-1)^\ell$ by Formula \ref{E:12} with $\ell =m$. 

This difference in sign causes a missing factor of $(-1)^m$ in the second and last formulas in Theorem 8.1 \cite{RRT},  
and so all the $\nu$-terms and spin numbers in \S9.1 of \cite{RRT} should change in sign. 
The spin number in \S9.2 of \cite{RRT} should not change because its sign was miscalculated. 
These corrections are consistent with the definitions of $\Delta_n^\pm$ in \cite{A-B, F, RRT}, and Theorems \ref{T:10.1} and \ref{T:10.2}.

We end with some minor corrections to our paper \cite{RRT}.
On l.\,$-6$ of p 15, $\omega$ should be ${\bf i}^m\omega$.
On l.\,$-8$ of p 20, $\mathbb C(2)$ should be $\mathbb H(2)$.
On l.\,$-8$ of p 26, [21, p. 175] should be [22, p. 175].
On l.\,19 of p 34, one of 3 should be one of 4, cf. Table \ref{ta:6}.

\subsection*{Acknowledgment}
We thank Gil Bor for motivating this paper by inquiring about the symmetry group of the spin structure on the Davis hyperbolic 4-manifold. 

We also thank the referee for carefully reading our paper, and making helpful suggestions that greatly improved the exposition in \S \ref{S:6}.



\begin{thebibliography}{99}

  
\bibitem{A-B} M.F. Atiyah and R. Bott, \emph{A Lefschetz fixed point formula for elliptic complexes: II. Applications}, Ann. Math. \textbf{88} (1968), 451-491. 

\bibitem{A-H} M.F. Atiyah and F. Hirzebruch, \emph{Spin-manifolds and group actions}, In: Essays on Topology and Related Topics, 
Springer-Verlag, New York, 1970, 18-28. 


\bibitem{atiyah-singer:III}
M.F. Atiyah and I.M. Singer, \emph{The index of elliptic operators: {III}}, 
 Ann. of Math. \textbf{87} (1968), 546--604. 
 
 \bibitem{C-P}
 C.J. Cummins and J. Patera, \emph{Polynomial icosahedral invariants}, J. Math. Phys. 29 (1988), 1736-1745.

\bibitem{davis:hyperbolic}
M.W. Davis, \emph{A hyperbolic {$4$}-manifold}, Proc. Amer. Math. Soc.
  \textbf{93} (1985), no.~2, 325--328.
  

  
\bibitem{F} T. Friedrich, \emph{Dirac Operators in Riemannian Geometry}, Graduate Studies in Math., Vol. 25, Amer. Math. Soc. Providence, RI, 2000.

\bibitem{G-L-T} M. Gromov, H.B. Lawson, and W. Thurston, {\it Hyperbolic 4-manifolds and conformally flat 3-manifolds}, Inst. Hautes \'Etudes Sci. Publ. Math. No. 68 (1988), 27–45.  

\bibitem{hitchin:spinors} N.J. Hitchin, \emph{Harmonic spinors}, Adv. Math. \textbf{14} (1974), 1--55.

\bibitem{K} S. Kobayashi, {\it Transformation Groups in Differential Geometry}, Springer-Verlag, Berlin, 1972.


\bibitem{L-M} H.B. Lawson JR. and M.-L. Michelsohn, \emph{Spin Geometry}, Princeton Univ. Press, Princeton, NJ, 1989. 


\bibitem{M-S} J.W. Milnor and J.D. Stasheff, {\it Characteristic Classes}, Ann. Math. Studies 76, Princeton Univ. Press, Princeton NJ, 1976. 

\bibitem{P-T} R.S. Palais, and C-l.Terng, {\it Critical Point Theory and Submanifold Geometry}, Lecture Notes in Math. {\bf 1353}, Springer-Verlag, Berlin, 1988. 

\bibitem{R}
J.G. Ratcliffe, \emph{Foundations of Hyperbolic Manifolds, Third Edition}.
Graduate Texts in Mathematics 149, Springer Nature Switzerland AG, 2019. 

\bibitem{ratcliffe-tschantz:davis}
J.G. Ratcliffe and S.T. Tschantz, \emph{On the {D}avis hyperbolic
  4-manifold}, Topology Appl. \textbf{111} (2001), 327--342.
  
\bibitem{RRT}
J.G. Ratcliffe, D. Ruberman, S.T. Tschantz, \emph{Harmonic spinors on the Davis hyperbolic 4-manifold}, 
J. Topol. Anal. {\bf 13} (2021), 699-737.

\bibitem{S}
P. Shanahan, {\it The Atiyah-Singer Index Theorem}, Lecture Notes in Math. {\bf 638}, Springer-Verlag, Berlin, 1978. 


\end{thebibliography}
\end{document}